\documentclass[12]{amsart}

\usepackage{amssymb}
\newtheorem{theorem}{Theorem}

\newtheorem{corollary}[theorem]{Corollary}
\newtheorem{lemma}[theorem]{Lemma}

\newtheorem{que}{Question}

\theoremstyle{definition}
\newtheorem{delfin}{Definition}

\newtheorem{remark}[theorem]{Remark}

\begin{document}
\title{How far can we go with Amitsur's theorem in differential polynomial rings?}

\author[Smoktunowicz]{Agata Smoktunowicz}
\address{Maxwell Institute for Mathematical Sciences, School of Mathematics,\newline   University of Edinburgh,\newline JCM Building, King’s Buildings, Mayfield Road
Edinburgh EH9 3JZ, Scotland, UK,\newline   E-mail: A.Smoktunowicz@ed.ac.uk}


\date{\today}

\begin{abstract} 
 A well-known theorem by S.A. Amitsur shows that the Jacobson radical of the polynomial ring
$R[x]$ equals $I[x]$ for some nil ideal $I$ of $R$. In this paper, however, we show that this is not the case for differential polynomial rings, by proving that there is a ring $R$ which is not nil and a derivation $D $ on $R$ such that the differential polynomial ring $R[x; D ]$ is Jacobson radical. We also show that, on the other hand, the Amitsur theorem holds for a differential polynomial ring $R[x; D]$, provided that $D$ is a locally nilpotent derivation and $R$ is an algebra over a field of characteristic $p>0$. The main idea of the proof  introduces a new way of embedding differential polynomial rings into bigger rings,  which we name platinum rings, plus a key part of the proof  involves the solution of matrix theory-based problems.
\end{abstract}

\maketitle


\section{Introduction} Let $R$ be a noncommutative associative ring. 
In 1956, S.A. Amitsur proved that the Jacobson radical of the polynomial ring $R[x]$ equals $I[x]$ for some nil ideal $I$ of $R$  \cite{lam}. Then in 1980, S. S. Bedi and J. Ram extended Amitsur's theorem to skew polynomial
rings of automorphism type \cite{bendi}. The question then  arises as to whether Amitsur's theorem also holds for differential polynomial rings; that is, whether  the Jacobson radical of $R[x; D ]$ equals $I[x; D ]$ for a nil ideal $I$ of $R$. 
 In 1975, D. A. Jordan \cite{Jordan} showed that 
  Amitsur's theorem holds for   differential polynomial rings $R[x; D ]$, provided that $R$ is a Noetherian ring with an identity, and in 1983, 
M. Ferrero, K. Kishimoto and K. Motose \cite{ferrero} showed that in the general case the Jacobson radical of $R[x, D ]$ equals $I[x, D ]$ for 
 an ideal $I$ of $R$ (and $I$ is nil if $R$ is commutative).  However, it remained an open question as to whether $I$ needs to be nil. We will answer this question in the negative, by proving the following theorem.
\begin{theorem}\label{main} Let $K$ be an arbitrary subfield  of  the algebraic closure of a finite field. There is an $K$-algebra $R$  and a derivation $D$ on $R$ such that $R$ is not nil and the algebra $R[x; D ]$ is Jacobson radical. 

In particiular 
there is a ring $R$ which is not nil and a derivation $D $ on $R$ such that the ring $R[x; D ]$ is Jacobson radical.
\end{theorem}
 Let $F$ be the algebraic closure of a finite field, and let $R$ be an $F$-algebra. Let $K$ be a subfield (possibly finite) of the field $F$, then $R$ is also a $K$-algebra. Moreover,  $R$ is  nil as a $K$-algebra if and only if $R$ is nil as an $F$-algebra; similarly 
$R[x; D ]$ is Jacobson radical as a $K$-algebra if and only if $R[x; D ]$ is Jacobson radical as an $F$-algebra.
 Therefore Theorem \ref{main} also holds in the case when $K$ is a finite field.

However, in the case when $D$ is a locally nilpotent derivation we are able to show the following.
\begin{theorem}\label{aggi}
 Let $F$ be a field of characteristic $p>0$, let $R$ be an $F$-algebra and $D$ be a derivation on $R$. If $D$ is a locally nilpotent derivation, then the Jacobson radical of the differential polynomial ring $R[x; D]$ equals $I[x]$ for some nil ideal $I$ of $R$. 
\end{theorem}   
In 1987,  J. Bergen, S.  Montgomery and D.S.  Passman showed that Amitsur's theorem also  holds for differential polynomial rings in the case where  $R$ is a polynomial identity algebra, and obtained  far-reaching related results for enveloping algebras of Lie algebras and crossed products  \cite{montgomery}. Surprising applications of derivations in  Lie algebras and  nil algebras
 were found by V. M. Petrogradsky,  I.P. Shestakov and  E. Zelmanov \cite{petr, zel, sh}.
 We also note that the Jacobson radical of a ring $R[x; D ]$ in the case when $R$ has no nil ideals was investigated by P. Grzeszczuk and J. Bergen  \cite{grzeszczuk}. For other results on such rings,  see  \cite{ali, yuan}. 
  Interesting results in the case where $R$ is a polynomial identity ring were obtained by J. Bell, B. Madill and F. Shinko in  \cite{Jason}, and  by B. Madill in \cite{blake}; for example, in \cite{Jason} it was shown that, if $R$ is a locally nilpotent ring satisfying a polynomial identity, then $R[x; D]$ is Jacobson radical. 
 This does not hold in general for an  arbitrary locally nilpotent ring $R$ (see \cite{sz}).

 It was shown by J. Krempa \cite{krempa} that the Koethe conjecture is equivalent to the assertion that polynomial rings over nil rings are Jacobson radical. Notice that logically Amitsur's result works  in the opposite direction to the Koethe problem. In the direction of Amitsur  it was shown by P. Nielsen and M. Ziembowski \cite{nielsen2} that $R[x; D]$  need not be prime radical provided that $R$ is a commutative  nil ring of bounded index  of nilpotency. Recall also that it was shown in \cite{ferrero} that if $R[x;D]$ is Jacobson radical then $R$ is Jacobson radical. 
 Notice that if  $R$ is an algebra over a field whose cardinality exceeds the dimension of $R$ as a vector space over $F$, and such that $R[x;D]$ is Jacobson radical, then $R[x; D]$ is nil and hence $R$ is nil \cite{amitsur}.
 The following questions remain open.

\begin{que}\label{x} Let $R$ be a ring without nil ideals, and $D$ a derivation on $R$; does it follow then that $R[x; D]$ is semiprimitive?
\end{que}
 Notice that by Theorem \ref{aggi} the answer to Question \ref{x} is affirmative when $R$ is an algebra over a field of positive characteristic and $D$ is locally nilpotent.
\begin{que} Let $F$ be a field of characteristic $p>0$, let $R$ be an $F$-algebra and $D$ be a locally nilpotent derivation on $R$. Suppose that $R[x; D]$ is nil. Does it follow that $R[x]$ is nil?
\end{que}
\begin{que}
 Are there examples as in Theorem \ref{main}, over the base field of characteristic $0$, or over non-algebraic extensions of finite fields? 
\end{que}
\begin{que}
 The examples constructed in Theorem \ref{main} are not finitely generated $F$-algebras. Is it possible to construct finitely generated examples?
\end{que}
\begin{que}
 Is Theorem \ref{aggi} valid over fields of characteristic zero?
\end{que}
\begin{que}
 It was proved in \cite{smok10} that any primitive ideal in $R[x]$, where $R$ is nil has the form $I[x]$. Is the analogous result valid for the differential polynomial setting?
\end{que}

Let $D $ be a derivation on a ring $R$. Recall that the differential polynomial ring $R[x; D]$ consists of all polynomials of the form $a_{n}x^{n}+\ldots +a_{1}x+a_{0}$, where $a_{i}\in R$ for
$i=0, 1, 2 \ldots ,n$. The ring $R[x; D]$ is considered with pointwise addition and multiplication given by $x^{i}x^{j}=x^{i+j}$ and $xa-ax=D(a)$, for all $a\in R$. For a given ring $A$ we denote by $A^{1}$ the usual extension with an identity of the ring $A$.
 In a non-unital algebra we assume that the ideal generated by a given set of elements contains these elements.
 The main idea of the proof is contained in the following result.
\begin{theorem}\label{inny100} Let $n>1$ be a natural number.
Let $F$ be an infinite field, and let $A'$ be a free (non-unital) $F$-algebra in generators $a_{1}, \ldots ,a_{n}$ and $x$, and let $A^{(*)}$ be the ideal of $A'$ generated by $a_{1}, \ldots , a_{n}$. Then $D(r)=xr-rx$ is a derivation on $A'$.
Let $P$ be the smallest subalgebra of $A'$ containing elements $a_{1}, \ldots , a_{n}$ and closed under the action of $D$.
 Let $I$ be an ideal in $A'$ with the property that  $\gamma _{t}(I)\subseteq I$
 for every $t\in F$,
 where $\gamma _{t}:A'\rightarrow A'^{1}$ is the  ring homomorphism such that $\gamma _{t}(a_{i})=a_{i}$ for all $i=1, \ldots ,n$ and $\gamma _{t}(x)=x+t$.
 Then the $F$-algebra \[A^{(*)}/I\cap A^{(*)}\]
 is isomorphic to the differential ring 
$\bar {P}[y; D']$, where $\bar {P}=P/I\cap P$ and $D'$ is the derivation on $\bar P$ such that $D'(p+(I\cap P))=xp-px+ (I\cap P)$ for every $p\in P$.
\end{theorem} Observe, on the other hand, that  a differential polynomial $F$-algebra $R[x;D]$ can be, in a natural way, embedded into the
 factor ring $<R,x>/I$, where  $<R,x>$ is the free product of $R$ and the polynomial ring $F[x]$ and $I$ is the ideal generated by relations $xr-rx-D(r)$. Notice that $\gamma _{t}(I)\subseteq I$ for  
every  $t\in F$,
 where $\gamma _{t}:<R,x>\rightarrow <R,x>^{1}$ is the  ring homomorphism such that $\gamma _{t}(r)=r$ for all $r\in R$ and $\gamma _{t}(x)=x+t$. This shows that differential polynomial rings have a presentation similar to the presentation from Theorem \ref{inny}.

An outline of the proof for Theorem \ref{main} now follows:
\begin{itemize}
\item Let $F$ be a field, and let $A'$ be a free algebra in generators $a,b,x$, and $A^{(*)}$ be the ideal of $A'$ generated by $a$ and $b$.
We introduce ideal $I^{(*)}$  in $A^{(*)}$ which is generated by entries of powers of some  matrices $X_{1}, X_{2}, \ldots $.
 It is then shown that $A^{(*)}/I^{(*)}$ is Jacobson radical. 

\item We introduce the platinum  ideal $L(I^{(*)})$ of $A^{(*)}$. We define $L(I^{(*)})$ to be the smallest ideal such that $I^{(*)}\subseteq L(I^{(*)})$ and $\gamma _{t}(I^{(*)})\subseteq L(I^{(*)})$
 for every $t\in F$,
 where $\gamma _{t}:A'\rightarrow A'^{1}$ is the  ring homomorphism such that $\gamma _{t}(a)=a$, $\gamma _{t}(b)=b$ and $\gamma _{t}(x)=x+t$.
\item It is then shown that  $A^{(*)}/L(I^{(*)})$ is isomorphic to some differential polynomial ring $Z[y;D].$
\item Since $A^{(*)}/I^{(*)}$ is Jacobson radical, then $A^{(*)}/L(I^{(*)})$ is Jacobson radical. It follows that $Z[y; D]$ is Jacobson radical.
\item Next we introduce Assumption $1$, and show that if $F$ is a field which is the algebraic closure of a finite field then Assumption $1$ holds.
\item    It is then shown that if Assumption $1$ holds then some subrings of $A^{(*)}/L(I^{(*)})$ are not nil, which implies that $Z$ is not nil. 
\item The last two sections contain matrix theory-based problems, which are an important part of the proof. 
\end{itemize}
For  general information on polynomial identity algebras we refer the reader to \cite{drensky} and \cite{rowen}, and for diferential polynomial rings over associative noncommutative rings to \cite{kharchenko, rowe} and \cite{Carl}. 
 We prove Theorem \ref{aggi} in Section $2$. 
 Sections $3-9$ and $10-17$ are mathematically independent of each other and hence can be considered separately (in Sections $3$-$9$ we prove Theorem \ref{main} under the Assumption $1$, and in Sections $10-17$ we prove Assumption $1$ for algebras over some fields). Theorem \ref{inny} is proved in Section \ref{platinum}.
\section{Proof of Theorem \ref{aggi}}

Let $R$ be a ring.  Recall that an element $r\in R$ is  quasi-invertible in $R$  if there is $s\in R$ such that $r+s+rs=r+s+sr=0$.
 As every ring can be embedded in a ring with an identity element, this can be written as $(1+r)(1+s)=1$.
 Such an element $s$ is called a quasi-inverse of $r$.
We start with the following well-known fact
\begin{lemma}\label{known2}
 Let $ Q$ be a ring, and let $a\in Q$ be quasi-invertible, and let $b,c\in Q$ be quasi-inverses of $a$; then $b=c$.  
\end{lemma}
\begin{proof} $Q$ is a subring of a ring  $Q^{1}$ with identity. Then $1+b=(1+b)((1+a)(1+c))=((1+b)(1+a))(1+c)=1+c$, so $b=c$.
\end{proof}

Let $F$ be a field of characteristic $p>0$ and let $R$ be an $F$-algebra. Let $D$ be a locally nilpotent derivation on $R$. 
  Let $a\in R$, then $D^{n}(a)=0$ for some $n$. Observe that using rule $x\cdot D^{n}(a)-D^{n}(a)\cdot x=D^{n+1}(a)$, it can be proved by induction that $D^{n}(a)=\sum_{i=0}^{n}\alpha _{i}x^{n-i}ax^{i}$ where $\alpha _{i}=(-1)^{i}{n \choose i}$ (it can also be inferred using rule $(z- q)^{n}=\sum _{i=0}^{n}\alpha _{i}z^{j-i}{q}^{i},$ where $z$ denotes multiplication from the left by $x$, and $ {q}$ multiplication from the right by $x$).
 Then $D^{p^{m}}(a)=x^{p^{m}}\cdot a-a\cdot x^{p^{m}}.$  Notice that the  binomial coefficients ${n \choose i}$ are well defined for fields of finite characteristic.
  
{\bf Proof of Theorem \ref{aggi}}:
Let notation be as above, and let $a\in R$. 
 $D$ is a locally nilpotent derivation, so there is $m$ such that $0=D^{p^{m}}(a)=x^{p^{m}}\cdot a-a\cdot x^{p^{m}}$. 
If $R[x; D]$ is Jacobson radical, then $ax^{p^{m}}$ is quasi-invertible in $R[x, D ]$. Let $s$ be the quasi-inverse of $ax^{p^{m}}$; then  $s=\sum_{i=0}^{n}a_{i}x^{i}$ for some $a_{i}\in R$. Let $S$ be a subring of $R$ generated by elements $a, a_{0}, a_{1}, \ldots , a_{n}$ and elements $D^{i}(a)$, $D^{i}(a_{0}), D^{i}(a_{1})\ldots , D^{i}(a_{n})$ for $i=1,2, \ldots $.
 Then $D$ is a derivation on $S$ and $S[x; D]$ is a subring of $R[x; D]$. Notice that element $ax^{p^{m}}$ 
 is quasi-invertible in $S[x; D]$. 
 Recall that $D$ is a locally nilpotent derivation, so there is $k>m$ such that $0=D^{p^{k}}(a_{i})=x^{p^{k}}a_{i}-a_{i}x^{p^{k}}$ for $0\leq i\leq n$.
 Then  $x^{p^{k}}$ commutes with all elements of $S$, since $x^{p^{k}}D^{j}(a_{i})-D^{j}(a_{i})x^{p^{k}}=D^{p^{k}+j}(a_{i})=0$.
 Therefore, $D$ is a nilpotent derivation on $S$, since $D^{p^{k}}(r)=x^{p^{k}}\cdot r-r\cdot x^{p^{k}}=0$ for every $r\in S$.
Notice that $S[x; D]$ is a subring of a ring $Q$, 
 where $Q$ is  
 the set of all series $\sum_{i=0}^{\infty}c_{i}x^{i}$ with $c_{i}\in S$ with natural addition and multiplication 
$xc-cx=D(c)$ for $c\in S$. The multiplication on $Q$ is well defined because $D$ is a nilpotent derivation on $S$.
  Recall that $x^{p^{m}}\cdot a-a\cdot x^{p^{m}}=0$, hence $(ax^{p^{m}})^{i}=a^{i}\cdot x^{i\cdot p^{m}}$. 
Observe that for $c=a\cdot x^{p^{m}}$ we have  $(1+c)(1-c+c^{2}-c^{3}+\ldots) =1$.
 Therefore, $a'=\sum_{i=1}^{\infty }(-1)^{i}c^{i}=\sum_{i=1}^{\infty}(-1)^{i}a^{i}x^{i\cdot p^{m}}$ is a quasi-inverse of $a\cdot x^{p^{m}}$ in $Q$. By Lemma \ref{known2}, we get $s=a'$, hence $a'\in S[x;D]$. It follows that 
$a^{i}x^{i\cdot p^{m}}=0$ for almost all $i$; hence $a$ is nilpotent.

\section{Definitions and the Jacobson radical}
 Let $F$ be a field. Throughout this paper we will assume that $F$ is a countable and infinite field. Notice in particular that the algebraic closure of any finite field is countable. 
 Let $A'$ be a free noncommutative $F$- algebra generated by elements $a, b$ and $x$; $A'$ is a free algebra in the category of non-unital algebras, so it does not contain elements with non-zero constant term. We assign gradation $1$ to elements $a$ and $b$ and  we assign gradation $0$ to element $x$. By $R$ we denote the subalgebra of $A'$ generated by $a$ and $b$, and by $A$ we denote a subalgebra of $A'$ generated by $ax^{i}, bx^{i}$ for $i\geq 0 $. Notice that $A=RA'+R$, and hence $A$ is a left ideal in $A'$. 
 By $A'(n)$  we will denote the linear space spanned by all elements with gradation  $n$ in $A'$.
 In general, if $T$ is a linear subspace of  $A'$, then we denote $T(n)=T\cap A'(n)$. In particular, $A(n)$ denotes 
 the linear space spanned by all elements with gradation  $n$ in $A$. For a given ring $Q$ we denote by $Q^{1}$ the usual extension with an identity of the ring $Q$.
  By $\langle x\rangle $ we will denote the ideal generated by $x$ in $A'$.
 
Denote  $ A^{(*)}=A+xA+x^{2}A+\ldots =\sum_{i=0}^{\infty }x^{i}A$, notice that $A^{(*)}$ is the ideal of $A'$ generated by $a$ and $b$. Given an ideal $I$ in $A$ 
 we denote $I^{(*)}=I+xI+x^{2}I+\ldots =\sum_{i=0}^{\infty }x^{i}I$.

 The following Lemma is a reformulation of Lemma $7.2$ in \cite{bull}. 
\begin{lemma}\label{p} Let $r\in A$. Then there is a matrix $X_{r}$ of some finite size with entries in $A(1)$ and such that for every $n>0,$  
  $r+Q_{r, n}$ is quasi-invertible in algebra $A/Q_{r, n}$ where $Q_{r, n}$  is the ideal generated by coefficients of matrix $X_{r}^{n}$ in $A$. If $r\in R$, then the quasi-inverse of $r+Q_{r, n}$ equals $s+Q_{r, n}$ for some $s\in R$. If $r\in \langle x\rangle $, then there is $\alpha (X_{r})$ such that $X_{r}^{i}$ has all entries in $\langle x\rangle $ for every $i>\alpha (X_{r})$. 
\end{lemma}
\begin{proof}
  To every $r\in A$ we can assign matrix $X_{r}$ of some finite size with entries in $A(1)$, like in Definition $7.1$ in \cite{bull}. Let $n$ be a natural number.  We can apply Lemma $7.2$ from \cite{bull} to $S=A(1)$  and $r=\sum_{i=1}^{\gamma }s_{i}$ with $s_{i}\in S^{i}=A(i)$. Recall that we used the following notation in  \cite{bull}, $v_{0}=1$
 and $v_{i}$ is the sum of all products $s_{j_{1}}s_{j_{2}}\ldots s_{j_{k}}$ where $j_{1}+\ldots +j_{k}=i$ and $k$ is arbitrary.
 Observe now that by Lemma $7.2$ in \cite{bull} $r$ is quasi-invertible in $A/Q(r,n)$.

 Observe also that if $r\in A\cap \langle x\rangle $, then for a sufficiently large $n$ matrix $X_{r}^{n}$ has all entries from  $A\cap \langle x\rangle $ by the Remark on page 925 in \cite{bull}.
 Notice also that all coefficients of $X_{r}^{n}$ are from $A(n)$.

 If $r\in R$,  then the entries of matrix  $Q_{r,n}$ are in $R$, hence by the above reasoning applied to ring $R$ instead of $A$ we get that $r+Q_{n}$ is quasi-invertible in $R/Q_{n},$ where $Q_{n}$ is the ideal generated by entries of $X_{r}^{n}$ in $R$. Observe that $Q_{n}\subseteq Q_{r,n}$. By Lemma \ref{known2},  $r+Q_{r,n}$ has a quasi-inverse of the form $s+Q_{r,n}$ for some $s\in R$, as required.
\end{proof}
 The field $F$  is countable, hence the set of elements of $A$ is countable (since $A$ is finitely generated). It follows that the set of all matrices  
 $X_{r}$ for $r\in A'$ is countable. We can ennumerate the matrices $X_{r}$ with either $r\in A\cap A'xA'$ or $r\in R$ as $X_{1}, X_{2}, \ldots $. 

The main result of this section is the following: 
\begin{theorem}\label{Jacobson} Let notation be as above, in particular let matrices $X_{1}, X_{2}, \ldots $ be as above.  Let $0<m_{1}<m_{2}< \ldots $
 be a sequence of natural numbers such that $20m_{i}$ divides $m_{i+1}$, and $m_{i}>\alpha (X_{i})$ (where $\alpha (X_{i})$ is as in Lemma \ref{p}). 
Let $S'_{i}$ be the linear space spanned by all entries of the matrix $X_{i}^{m_{i}}$ and  let
\[S_{i}=\sum_{j=1}^{\infty } A(j\cdot 20m_{i}-2m_{i})S'_{i}A(m_{i})A^{1}.\]
 Let $I$ be the ideal of $A$ generated by the entries of matrices $X_{k}^{30m_{k}}\cdot x^{i}$ for all $k>0$ and all $i\geq 0$
(where the multiplication $X_{k}^{30m_{k}}\cdot x^{i}$ is component-by-component).
  Then $I$  is a homogeneous ideal of $A$, $I$ is contained in $\sum_{i=1}^{\infty }S_{i}$  and  $A/I$ is Jacobson radical.  Moreover, $IA'\subseteq A'$, and 
 if $g+h\in I$ and $g\in R$ and $h\in A\cap \langle x\rangle $ then $g,h\in I$.
\end{theorem}
\begin{proof}  By Lemma \ref{p}  all entries of matrices $X_{k}$ are in $A(1)$, hence  $I$ is a homogeneous ideal of $A$.  Let $k>0$, observe first that the ideal $I_{k}$ of $A$ generated by entries of the matrices  
$X_{k}^{30m_{k}}$  is contained in the subspace $S_{k}$.
  It follows because entries of every matrix $X_{k}$ have degree one. Namely, if $n>i+2m_{k}$ then  every entry of matrix
 $X_{k}^{n}$ belongs to $A(i)S'_{k}A(m_{k})A$ for every $0\leq i$. Similarly every entry of matrix
 $X_{k}^{n}\cdot x^{i}$ belongs to $A(i)S'_{k}A(m_{k})AA'\subseteq A(i)S'_{k}A(m_{k})A.$  
 Observe also that, by Lemma \ref{p}, all elements $r\in R$ and all elements $r\in A\cap \langle x\rangle $ are quasi-invertible in $A/I$.
Notice also that $IA'\subseteq I$, as $X_{k}^{n}\cdot x^{i}\cdot r$ has entries in $I$ for every $r\in A'$.

 We will now show that for every $r\in A$ element $r+I$ is quasi-invertible in $A/I$.
Let $r=u+v$, where $u\in R$ and $v\in A\cap \langle x\rangle $. 
 Since $u\in R$ then by Lemma \ref{p}, there is $u'\in R$ such that 
$(1+u)(1+u')+I=1+I$. Notice that element $(1+r)+I$ has right inverse if and only if element 
$(1+r)(1+u')+I$ has right inverse in $(A/I)^{1}$. We see that $(1+r)(1+u')+I=(1+u+v)(1+u')+I=1+v(1+u')+I$.
 By assumption $1+v(1+u')+I$  has right inverse by Lemma \ref{p}, because $v(1+u')\in A\cap \langle x\rangle $ since $v\in A\cap \langle x\rangle $. 
It follows that $1+r+I$ has  a right inverse in $(A/I)^{1}$.
In a similar way we show that $1+r+I$ has  a left inverse in $(A/I)^{1}$. Therefore $r+I$ is quasi-invertible in $A/I$ (similarly as in Lemma \ref{known2}).

The last assertion from the thesis of our theorems follws because $m_{i}> \alpha (X_{i})$, and so the ideal generated by entries of matrix $X_{i}^{30m_{i}}$
 is either contained in $\langle x\rangle $ or is generated by elements from $R$. 
\end{proof}
 
Recall that $ A^{(*)}=\sum_{i=0}^{\infty }x^{i}A$. Notice that $A'= A^{(*)}+xF[x]$ where $F[x]$ is the polynomial ring over $F$ (since $A'$ does not contain elements with non-zero constant terms). 
   Given an ideal $I$ in $A$ we denote  $I^{(*)}=I+xI+x^{2}I+\ldots =\sum_{i=0}^{\infty }x^{i}I$.
 
\begin{lemma}\label{mily} Let $ A^{(*)}$ be as above.
 Let $I$ be an ideal in $A$ which is also a right ideal in $A'$ (so $IA'\subseteq I$). Let $I^{(*)}=I+xI+x^{2}I+\ldots =\sum_{i=0}^{\infty }x^{i}I$. Then $I^{(*)}$ is an ideal in $A^{(*)}$ and $I\cap R=I^{(*)}\cap R$. In addition if $r+I$ is not a nilpotent in $A/I$ for some $r\in R$, then $r+I^{(*)}$ is not a nilpotent in $A^{(*)}/I^{(*)}$.
 Moreover, if $A/I$ is Jacobson radical then $A^{(*)}/I^{(*)}$ is Jacobson radical.
\end{lemma}
\begin{proof} Assume that $A/I$ is Jacobson radical.  Observe that $I^{(*)}$ is a two-sided ideal in $A'$. From Lemma $4.1$ on page $50$ in \cite{lam} (by interchanging the left and the right side), we see that an element $y$ is in the Jacobson radical of $ A'/I^{(*)}$ if $yq+I^{(*)}$ is quasi-invertible in $A'/I^{(*)}$ for every $q\in A'$. Clearly, if $y\in A$ then $yq\in A$ for every $q\in A'$. By assumption, if $r=yq\in A$ then 
$r+I$ is quasi-invertible in $A/I$ and hence $r+I^{(*)}$ is quasi- invertible in $A'/I^{(*)}$. Therefore every element $a+I^{(*)}$ for $a\in A$ is in the Jacobson radical of $A'/I^{(*)}$.
 Recall that the Jacobson radical is a two sided ideal.  Therefore every element $a'+I^{(*)}$ for $a'\in A^{(*)}$ is in the Jacobson radical of $A'/I^{(*)}$. Observe that if $b+I^{(*)}$ is a quasi-inverse of $a'+I^{(*)}$,
 then $b=-a'b-a'\in A^{(*)}+I^{(*)}\subseteq A^{(*)}.$ Therefore we can assume that $b\in A^{(*)}$. It follows that $A^{(*)}/I^{(*)}$ is Jacobson radical.  

We will now show that $I\cap R=I^{(*)}\cap R$. Let $T_{i}=x^{i}(RA'+R)=x^{i}A$ for $i\geq 0$. Recall that $A'$ is a free algebra, and $x\notin R$.  
 Therefore if $0\neq t_{i}\in T_{i}$ for $i\geq 0 $ then elements 
$t_{0}, t_{1}, \ldots $ are linearly independent over $F$.

Let $i\in I$, then  $i=i_{0}+i_{1}+\ldots +i_{n}$ for some  $i_{0}\in I, i_{1}\in xI, \ldots , i_{n}\in x^{n}I$.
 Observe that since $I\subseteq A$ then $i_{j}\in T_{j}$ for $j=1,2, \ldots , n$.  
 If $i\in R$ then  $i-i_{0}=i_{1}+i_{2}+\ldots +i_{n}$. Notice that $i-i_{0}\in T_{0}$.  The above observation  on elements $t_{i}$ implies that elements 
$i-i_{0}, i_{1}, i_{2}, \ldots , i_{n}$ are all equal zero, so $i=i_{0}\in I$. 
 Therefore $I\cap R=I^{(*)}\cap R$.

Suppose now that $r+I^{(*)}$ is nilpotent in $ A^{(*)}/I^{(*)}$. Then $r^{n}\in I^{(*)}$ for some $n$, so $r^{n}\in I$ by the above,
 and so $r+I$ is nilpotent in $A/I$. 
\end{proof}

\section{ Platinum ideals and platinum subspaces}\label{platinum}

 In this section we introduce platinum spaces, which will be useful for constructing examples of differential polynomial rings.
 Let notation be as in the previous sections, in particular $A'$ is generated by elements $a,b$ and $x$, and $R$ is generated by elements $a$ and $b$. 

\begin{delfin}\label{ppp}
  Let $P$ be the smallest subring of $A'$ satisfying the following properties.
\begin{itemize}
\item $R\subseteq P$
\item If $c\in P$ then $xc-cx\in P$
\end{itemize} 
\end{delfin}

 For a $c\in R$ define $D (c)=xc-cx$. Then $D $ is a derivation on $P$. Therefore we can consider the differential polynomial ring 
 $P[y; D]$ where $yc-cy=D (c)$ for $c\in P$.

\begin{remark} Another way to define $P$ is to first to define the inner derivation $D$ associated with element $x$ as $D(r)=xr-rx$ for $r\in A'$, and then define $P$ as the intersection 
 of all subrings of $A'$ which contain $R$ and are closed under the action of $D$.
 It is then clear that $P\subseteq A'$.
\end{remark}

 Recall that $A'$ is a free algebra with free generators $a, b, x$.
Let $q\in F$ then let $\gamma _{q}: A'\rightarrow A'^{1}$ be a ring homomorphism such that 
\[ \gamma _{q}(a)=a, \gamma _{q}(b)=b, \gamma _{q}(x)=x+q.\]  

\begin{lemma}\label{staly}
 Let $q\in F$, then $ \gamma _{q}(p)=p$ for every $p\in P$.
\end{lemma}
\begin{proof}
 We proceed by induction using the definition of $P$. Observe that $ \gamma _{q}(r)=r$ for every $r\in R$. 
 If $u,v\in P$ and $ \gamma _{q}(u)=u$ and $ \gamma _{q}(v)=v$ then $ \gamma _{q}(u+v)=u+v$ and 
$ \gamma _{q}(uv)=uv$ and $ \gamma _{q}(xu-xu)=(x+q)\gamma _{q}(u)-\gamma _{q}(u)(x+q)=xu-ux$. 
\end{proof}
\begin{lemma}\label{useful1}  Let  notation be as in Definition \ref{ppp} and let $F$ be an infinite field.  Let $f: P[y; D]\rightarrow A'$ be a $F$-linear mapping such that 
$f(p)=p$  and $f(py^{i})=px^{i}$  for $p\in P$, $i\geq 1$. Then $f$ is injective and $f$ is a homomorphism of rings.
\end{lemma}
\begin{proof}  
 Observe that $P$ embeds into $A'$ in a natural way as a subring, hence $f(p)$ is well defined as a ring homomorphism for $p\in P$. 
 Every element of $ P[y; D]$ can be uniquely written as a linear combination of elements $py^{i}$ with $p\in P$, hence $f$ is well defined as a linear mapping. We will show that $f$ is a ring homomorphism. Notice first that binomial coefficients $n\choose i$ are well definied in fields of finite characteristic.
 Observe that if $p,q\in P[y; D]$ then $py^{n}q=p\sum _{i=0}^{n}{n\choose i}D ^{i}(q)y^{n-i}$.

 Therefore 
$f(py^{n}\cdot qy^{j})=f(p\sum _{i=0}^{n}{n\choose i}D^{i}(q)y^{n-i}y^{j})=p\sum _{i=0}^{n}{n\choose i}D ^{i}(q)x^{n-i+j}.$
 On the other hand, 
$f(py^{n})f(qy^{j})=px^{n}\cdot qx^{j}=p\sum _{i=0}^{n}{n\choose i}D ^{i}(q)x^{n-i+j}.$
 Consequently,
 $f(py^{n}\cdot qy^{j})=f(py^{n})f(qy^{j})$.

We need to show that the kernel of $f$ is zero. Suppose that $f(\sum_{i=0}^{n}c_{i}y^{i})=0$ for some 
 $c_{i}\in P$. 
Since $f(\sum_{i=0}^{n}c_{i}y^{i})=0$ then $\sum_{i=0}^{n}c_{i}x^{i}=0$.

If $c=0$ in $A'$ then clearly $\gamma _{q}(c)=0$ for every $q\in F$ (since $\gamma _{q}$ is an homomorphism of rings).
 By Lemma \ref{staly} we get $0=\gamma_{q}(\sum_{i=0}^{n}c_{i}x^{i})=\sum_{i=0}^{n}c_{i}(x+q)^{i}$.
 We can write such equations for pairwise distinct  elements $q_{1}, q_{2}, \ldots , q_{n+1}\in F$ and then write 
 them as $(c_{0}, c_{1}, \ldots , c_{n})M=0$ where $M$ is a matrix with $i$-th column equal to 
$(1, (x+q_{i}), (x+q_{i})^{2}, \ldots , (x+q_{i})^{n}    )^{T}.$ Observe that $M$ is a transposition of a  Vandermonde matrix, and hence the determinant of $M$ is 
$det(M)=\prod_{i>j}(q_{i}-q_{j})\in F$. Hence there is matrix $N$ in $F[x]$ such that $MN=Id\cdot det(M)$, where $Id$ is the identity matrix. It follows that  $(c_{0}, c_{1}, \ldots , c_{n})M=0$ implies $(c_{0}, c_{1}, \ldots , c_{n})det(M)=0$, and so $c_{0}=c_{1}=\ldots =c_{n}=0$.
\end{proof}  
 Recall that $ {A}^{(x)}=A+xA+x^{2}A+\ldots $.
\begin{lemma}\label{writing}
Every element of  $A^{(*)}$  can be uniquely written in the form $P+Px+Px^{2}+\ldots ,$ 
moreover since  $A=RF[x]+RA^{(*)}$ every element of $A$ can be written in the form $R[x]+RP+RPx+RPx^{2}+\ldots $ 
(notice also that $RP\subseteq P\cap A$). 
\end{lemma}
\begin{proof} It can be shown by induction on $n$ that $x^{n}P\subseteq P+Px+Px^{2}+\ldots =\sum_{i=0}^{\infty }Px^{i}$. Next observe that the set $P+Px+Px^{2}+\ldots $ is closed under multiplication and addition, hence it is a subring of $A'$ containing $x^{i}R$ and $Rx^{i}$ for every $i$, hence it contains $A^{(*)}$.

 Suppose now that some elements from $P+Px+Px^{2}+\ldots $ are linearly dependent over $F$; then $\sum _{i=0}^{n}p_{i}x^{i}=0$ for some $p_{i}\in P$. As $A'$ is a free algebra and $\gamma _{t}$ is a ring homomorphism for $t\in F$ 
 then  $\gamma _{t}(\sum _{i=0}^{n}p_{i}x^{i})=0$. By Lemma \ref{staly} we get $\sum _{i=0}^{n}p_{i}(x+t)^{i}=0.$
 We can write such equations for different elements $t=q_{1}, t= q_{2}, \ldots , t=q_{n+1}\in F$ and then 
  write these equations as $(p_{0}, p_{1},\ldots , p_{n})M=0$, where $M$ is a matrix with $i$-th column equal to 
$(1, (x+q_{i}), (x+q_{i})^{2}, \ldots , (x+q_{i})^{n})^{T}.$
Notice that $M$ is a transposition of  a Vandermonde matrix and the determinant of $M$ is det$(M)=\prod_{i>j}(q_{i}-q_{j})\in F$.  Therefore $M$ is invertible, with $MN=$Id for some matrix $N$ with entries in $F[x]$ (where Id denotes the identity matrix). Hence  $(p_{0}, \ldots , p_{n})M=0$ implies  $(p_{0},
 \ldots , p_{n})=(p_{o},\ldots , p_{n})MN=0$, so $p_{0}=p_{1}=\ldots =p_{n}=0$, as required.
\end{proof}

\begin{delfin}\label{subspace}
 Let $S$ be a linear subspace in $A'$. We will say that $S$ is a {\em platinum space}  if $\gamma _{q}(S)\subseteq S$ for every $q\in F$. 
\end{delfin}
\begin{delfin}
 Let $I$ be an ideal in $A'$. We will say that $I$ is a {\em platinum ideal} if  $\gamma _{q}(I)\subseteq I$ for every $q\in F$. In particular, a platinum ideal is an ideal which is a platinum subspace of $A'$. 
 
 We will  say that $A'/I$ is a {\em platinum ring} if $I$ is a platinum ideal of $A'$.   
\end{delfin}
\begin{remark}\label{before} Let $I$ be an ideal in $A'$ and let ${\bar I}=I\cap P$. 
 Observe that $\bar I$ is an ideal in $P$ and if $c\in \bar I$ then $xc-cx\in \bar {I}$ (because $xc, cx\in I$).

 Let $I$ be a platinum ideal in $A'$.  Denote $\bar {I}=I\cap P$. For $p\in P$ we define $D (p)=xp-px$.
 Then $P/\bar {I}$ is a ring with the derivation $D (c+\bar {I})=D(c)+\bar {I}$ where $D(c)=xc-cx$.
  By the elementary {\em Second Isomorphism Theorem} ring  $P/\bar {I}$ can be embedded in $A'/I$ via the mapping $h: P/\bar {I}\rightarrow A'/I$, where $h(c+\bar {I})=c+I$ for $c\in P$ (see \cite{Jordan} for some related results).

\end{remark}

\begin{theorem}\label{useful2} Let $F$ be an infinite field.  Let $I$ be a platinum ideal in $A'$ then $I\subseteq A^{(*)}$.  Denote $\bar {I}=I\cap P$.
Let $P^*=(P/\bar {I})[y; D]$ be the differential polynomial ring with $y(c+\bar {I})-(c+{\bar I})y= D (c+\bar {I})$ where $D (c+\bar {I})=xc-cx+\bar {I}$, for $c\in  P$. Then the mapping 
$f:P^{*}\rightarrow A'/I$ given by $f(p+\bar {I})=p+I$  and $f((p+\bar {I})y^{i})=px^{i}+I$ for $p\in P$, 
 is an injective homomorphism of rings; moreover  the image of $P^{*}$ equals $A^{(*)}/I$. Therefore, $P^*$ can be embedded into $A'/I$.
\end{theorem}
\begin{proof} We will first show that $I\subseteq A^{(*)}$ and hence $I$ is an ideal in $A^{(*)}$. Observe that $A'=A^{(*)}+xF[x]$, since $A'$ doesn't contain elements with constant terms. Let $i=u+v\in I$ where $u\in A^{(*)}$ and $v\in xF[x]$. We will show that $v=0$; suppose on the contrary that $v\neq 0$. Notice that since $F$ is an infinite field then $\gamma _{t}(v)$ contains a non-zero constant term from $F$  for some $t\in F$. On the other hand $A'$ doesn't contain any elements with 
non-zero constant terms, hence $\gamma _{t}(v)\notin A'$. Therefore $\gamma _{t}(i)=\gamma _{t}(u+v)\notin A'$, a contradiction since $I$ is a platinum ideal. 

 We will now show that the image of $P^{*}$ is $A^{(*)}/I$.  Notice that $P^{*}=(P/{\bar I})+(P/\bar {I})y+(P/\bar {I})y^{2}\ldots =\sum_{i=0}^{\infty }(P/\bar {I})y^{i}.$  Consequently by the definition of mapping $f$  the image of
 $P^{*}$ in $A'/I$ equals $(P+I)+(Px+I)+\ldots =\sum _{i=0}^{\infty }Px^{i}+I$.   Clearly $P+Px+\ldots \subseteq A^{(*)}$, since $A^{(*)}$ equals the ideal of $A'$ generated by $a$ and $b$. 
 By Lemma \ref{writing} we have  $A^{(*)}=\sum_{i=0}^{\infty }Px^{i}$, hence $A^{(*)}/I=$Im$(P^{*})$.

 We will now show that $f$ is an injective homomorphism of rings. By Remark \ref{before}, $f$ is well defined as a ring homomorphism on $P/\bar {I}$. 
Every element of $P^{*}$ can be uniquely written as a linear combination of elements $py^{i}$ with $p\in P$, $i\geq 0$, so $f$ is well defined as a linear mapping on $P^{*}$.
 To check that $f$ is a ring homomorphism we proceed similarly as in Lemma \ref{useful}.

We need to show that the kernel of $f$ is zero. Suppose that $f(\sum_{i=0}^{n}(c_{i}+\bar {I})y^{i})=0$ where 
 $c_{i}+\bar {I}\in P/\bar {I},$ $c_{i}\in P$. By the definition of $f$ we have  $\sum_{i=0}^{n}(c_{i}x^{i}+I)=0+I$ in $A'/I$, hence $\sum_{i=0}^{n}c_{i}x^{i}\in I$.
 Since $I$ is a platinum ideal we get that  $\gamma _{q}(\sum_{i=0}^{n}c_{i}x^{i})\in I$ for every $q\in F$. 
 By Lemma \ref{p}, that implies $\sum_{i=0}^{n}c_{i}(x+q)^{i}\in I$.
 Write such equations for different elements $q_{1}, q_{2}, \ldots , q_{n+1}\in F$, and then write them as $(c_{0}, c_{1}, \ldots ,c_{n+1})M=Q$ where $M$ is a matrix with $i$-th column equal to 
$(1, (x+q_{i}), (x+q_{i})^{2}, \ldots , (x+q_{i})^{n} )^{T}$ and $Q$ is a vector with all entries from $I$.
 Observe that $M$ is a transposition of a  Vandermonde matrix and hence the determinant of $M$ is 
det$(M)=\prod_{i>j}(q_{i}-q_{j})\in F$. Therefore, $M$ is invertible, with $MN=$Id for some matrix $N$ with entries in $F[x]$ (where Id is the identity matrix).
 Hence $(c_{0}, c_{1}, \ldots , c_{n})M=Q$ implies $(c_{0}, c_{1}, \ldots , c_{n})=QN$. 
Since $QN$ is a vector with all entries in $I,$ then $c_{0}, c_{1}, \ldots ,c_{n}\in I$.
 Since $c_{i}\in P$ then $c_{i}\in P\cap I=\bar {I}$ for $i=0, 1, 2, \ldots , n$. Therefore $c_{i}+\bar {I}=0+\bar {I}$ for every $i\leq n+1$, and so $\sum_{i=0}^{n}(c_{i}+\bar {I})y^{i}=0$, as required.   
\end{proof}  

{\bf Proof of Theorem \ref{inny100}}.
 Notice that Theorem \ref{useful2} is a special case of Theorem \ref{inny100} for $a_{1}=a$ and $a_{2}=b$. 
 Observe that  the number of generators of $A$
 doesn't influence the proof of Theorem \ref{useful2}, so the proof of Theorem \ref{inny100} is the same as the proof of Theorem \ref{useful2}.

\begin{delfin} For an element $r\in A'$ we define $L(r)$=span$_{F}\{\gamma _{t}(r):t\in F\}$.
 Given a linear space $S\subseteq A'$ we define $L(S)=$span$_{F}\{L(r):r\in S\}$.  
 Note that $L(S)$  is the linear space spanned by all elements $\gamma _{t}(s)$ for $t\in F$, $s\in S$. 
\end{delfin}
\begin{lemma}\label{jeden} Let $S$ be a linear space, then $L(S)$ is the smallest platinum space containing $S$.
\end{lemma}
\begin{proof}
 If $S\subseteq Q$ and $Q$ is a platinum space then $\gamma _{t}(S)\subseteq Q$ for every $t\in F$. Therefore $L(S)\subseteq Q$. 
 We need to show that $L(S)$ is a platinum space. Let $s\in L(S)$, then $s=\sum _{t\in W }\gamma _{t}(s_{t})$ where $W$ is a finite subset of $F$ and $s_{t}\in S$.
 Let $k\in F$, then $\gamma _{k}(s)=\sum _{t\in W}\gamma _{k}(\gamma _{t}(s_{t}))=\sum _{t\in W}\gamma _{k+t}(s_{t})\in L(S)$.
\end{proof}
\begin{lemma}\label{two}
Let $S\subseteq A'$ be a platinum space, then $S\subseteq A^{(*)}$. Suppose that $sx\in S$  for every $s\in S$. Then $S= S'+S'x+S'x^{2}+\ldots $ where $S'=P\cap S$. 
\end{lemma}
\begin{proof} The proof that $S\subseteq A^{(*)}$ is the same as the proof that 
 $I\subseteq A^{(*)}$ in Theorem \ref{useful2}.
 Observe first that $S'x^{i}\subseteq S$ for every $i$. We will show that $S\subseteq S'+S'x+S'x^{2}+\ldots =\sum_{i=0}^{\infty }S'x^{i}.$ Let $r\in S$. By Lemma \ref{writing} we have  $r=\sum_{i=0}^{n}p_{i}x^{i}\in S$ for some $p_{i}\in P$. We will proceed by induction on $n$. If $n=0$  then 
 $r=p_{0}\in P\cap S$, as required. Suppose now that $n>0$ and  the result holds for all numbers smaller than $n$, and we will show that it holds for $n$. 
  Because $A$ is a platinum subspace, for every $\alpha \in F$ we have
$\sum_{i=0}^{n }p_{i}(x+\alpha )^{i}\in S$. 
Let $d=\sum_{i=0}^{n}p_{i}(x+\alpha )^{i}-r,$ then $d\in S$.
 Observe that $d=\sum_{i=0}^{n}(p_{i}(x+\alpha )^{i}-p_{i}x^{i})=\sum_{i=0}^{n-1}d_{i}x^{i}$ for some $d_{i}\in P$. By the inductive assumption $d_{0}\in S$, hence  
$\sum_{i=1}^{n}p_{i}\alpha ^{i}=d_{0}\in P\cap S$.
 This holds for every $\alpha \in F$.
By the Vandermonde matrix argument, we get $p_{1}, \ldots , p_{n}\in P\cap S$. Moreover,  $p_{0}=r-\sum_{i=1}^{n}p_{i}x^{i}\in  S$ and since $p_{0}\in P$ then  $p_{0}\in P\cap S$. 
\end{proof}  
 


\section{ Linear mappings $f$  and  $G$}

Let $A^{*}$ be the subalgebra of $A$ generated by elements $ax^{i}ax^{j}$ and $bx^{i}bx^{j}$ for all $i,j\geq 0$.
 Notice that the notation $A^{*}$ is distinct from the notation $A^{(*)}$, denoting different objects.
 
 Let $B'\subseteq A(2)$ be the linear $F$- space spanned by elements $ax^{i}bx^{j}$ and $bx^{i}ax^{j}$ for all $i,j\geq 0$.
 Let $B=\sum_{i=0}^{\infty }A(2i)B'A$. Observe that $A=A^{*}+B$ and $A^{*}\cap B=0$.

 By $F[x]$ we will denote the polynomial ring in variable $x$ over $F$.  Given a linear mapping $f$ by $\ker (f)$ and Im $(f)$ we denote the kernel and the image of $f$.

\begin{lemma}\label{one} Let $m$ be an even number and let $S\subseteq A^{*}(m)$ be a platinum space such that $sx\in S$ if $s\in S$. Then there is a linear mapping  
 $f:A^{*}(m)\rightarrow A^{*}(m)$ such that 
\begin{itemize}
\item [1.] $\ker (f)=S$,
\item [2.] $f(px^{i})=f(p)x^{i}$ for every $p\in A^{*}(m)\cap P$,
\item [3.] $f(p)\in P$ for every $p\in A^{*}(m)\cap P$.
\item [4.] Moreover, for every $s\in A^{*}(m)$ and  $t\in F$, 
\[f(\gamma _{t}(s))=\gamma _{t}(f(s)).\] 
\item[5.] There is a linear space $E\subseteq A^{*}(m)$ such that 
 $f(r)=r$ for $r\in E$, and $E\oplus S=A^{*}(m).$ Moreover Im $(f)\oplus  \ker (f)=A^{*}(m).$
\end{itemize}
\end{lemma}
\begin{proof} By Lemma \ref{two}, $S=S'+S'x+\ldots $ where $S'=S\cap P\subseteq A^{*}(m)$. By Zorn's lemma, there exists a maximal linear subspace $Q$ of $A^{*}(m)\cap P$ such that 
 $S'\cap Q=0$.  Observe that then $Q+S'=A^{*}(m)\cap P$ and that $Q$ is a platinum space (as every subspace of $P$ is a platinum space, by Lemma \ref{staly}). Define $f(r)=0$ for $r\in S'$ and  $f(r)=r$ for $r\in Q$. Observe that $A^{*}(m)=\sum _{i=0}^{\infty }(A^{*}(m)\cap P)x^{i}$ by Lemma \ref{two}.
 Define $f(px^{i})=px^{i}$ for $p\in A^{*}(m)\cap P$.  
 By Lemma \ref{writing}, $f$ is a well defined linear mapping.  Notice that $E=Q+Qx+Qx^{2}+\ldots $ satisfies $(5)$. Observe that $f(s)=0$ for every $s\in S$. If $r=r_{1}+r_{2}$ with $r_{i}\in S$ and $r_{2}\in Q+Qx+Qx^{2}+\ldots $ then $f(r)=r_{2}$, so the kernel of $f$ equals $S$.  The image of $f$ is $E$, so (5) holds.

 We will now show that $f(\gamma _{t}(s))=\gamma _{t}(f(s))$.  
Let $s\in S$; then $s=\sum_{i} p_{i}x^{i}$ for some $p_{i}\in S'$.
Then $f(s)=\sum _{i}f(p_{i})x^{i}$. 
 Since $S$ is a platinum space then 
$\gamma _{t}(s)\in S$ for any $t\in F$. 
 Observe that by the definition of $\gamma _{t}$  we get 
\[\gamma _{t}(s)=\sum _{i} p_{i}(x+t)^{i},\]
 since $\gamma _{t}(p)=p$ for $p\in P$ by Lemma \ref{staly}.
 Therefore
$f(\gamma _{t}(s))=\sum _{i}f(p_{i})(x+t)^{i}$.
  
Observe now that \[\gamma _{t}(f(s))=\gamma _{t}(\sum _{i}f(p_{i})x^{i})=\sum_{i} f(p_{i}) (x+t)^{i}\] (by Lemma \ref{staly}, since $f(p_{i})\in P$).
 It follows that
$f(\gamma _{t}(s))=\gamma _{t}(f(s))$.
\end{proof}

\begin{remark} Notice that the fifth statement of Lemma \ref{one} can be also formulated by saying that the short exact sequence 
$0\rightarrow S\rightarrow A^{*}(m)\rightarrow E\rightarrow 0$ induced by $f: A^{*}(m)\rightarrow $ Im$(f)=E$ is split  by the inclusion map section $E\rightarrow A^{*}(m)$. 
\end{remark}
For a matrix $M$, let $S(M)$ be the linear space spanned by all entries of $M$, and $L(M)$ be the linear space spanned by all  matrices $\gamma _{t}(M)$ for $t\in F$ ( where if $M$ has entries $m_{i,j}$ then 
$\gamma _{t}(M)$ has respective entries $\gamma _{t}(m_{i,j})$). Observe that $L(S(M))=S(L(M))$. 

\begin{delfin}\label{second} (Definition of mapping $G$)
 Let $m$ be an even number and let $f:A^{*}(m)\rightarrow A^{*}(m)$ be a linear mapping satisfying properties (1)--(5) from Lemma \ref{one} (for some platinum space $S$).
 Define a linear mapping $G: A^{*}(10m)\rightarrow A^{*}(10m)$ as follows:

 If $v_{1}, \ldots , v_{10}\in A^{*}(m)$ are monomials (products of generators)  and  
  $v=v_{1}v_{2}\ldots v_{10},$ then we define   

\[G(v)=G(v_{1}\ldots v_{10})=v_{1}v_{2}\ldots v_{9}f(v_{10}).\]

 We can extend the mapping $G$ by linearity to all elements of $A^{*}(10m)$.

For every natural number $j>0$ we  extend  the mapping $G$ to the linear mapping $G:A^{*}(j\cdot 10m)\rightarrow A^ {*}(j\cdot 10m)$ in the following way:
 if $w=w_{1}\ldots w_{j}$ where $w_{i}$ are monomials and $w_{i}\in A^{*}(10m)$, 
 then  we define
\[G(w)=G(w_{1})G(w_{2})\ldots G(w_{j}).\]
 We can then  extend the mapping $G$ by linearity to all elements of $A^{*}(j\cdot 10m)$.

 Moreover, we  can also  extend the mapping $G$ to matrices with entries in $A^{*}(j\cdot 10m)$,
 so if $M$ has entries $a_{i,j}$ then $G(M)$ has respectively entries $G(a_{i,j})$. In similar fashion we can extend the mappings $f$ and $\gamma _{t}$ to matrices.
\end{delfin}
\begin{lemma}\label{G5} Let $m, G$ be as in the definition of the mapping $G$ above. Let $m_{i}, n$ be natural numbers such that  $n$ divides $m$, and $10m$ divides $m_{i}$. Let $M_{i}$ be a matrix with entries in $A^{*}(n)$;
 then 
\[L(S(G({M_{i}}^{m_{i}\over n})))=G(L(S({M_{i}}^{m_{i}\over n}))).\]
\end{lemma} 
\begin{proof} Recall that $S(M)$ denotes the linear space spanned by entries of matrix $M$, hence $S(L(M))=L(S(M))$ and $S(G(M))=G(S(M))$ for any matrix $M$ with entries in $C$.
 Consequently it is sufficient to show that $L(G({M_{i}}^{m_{i}\over n}))=G(L({M_{i}}^{m_{i}\over n}))$.
 We will first show that 
$G(L(M^{10}))=L(G(M^{10}))$ for any $M$ with entries in $A^{*}(m)$.
 Let $t\in F$, then $G(\gamma _{t}(M^{10}))=G(\gamma _{t}(M^{9})\gamma _{t}(M))=(\gamma _{t}(M))^{9}f(\gamma _{t}(M))$. 
 By assertion (4) from Lemma \ref{one} we have $f(\gamma _{t}(M))=\gamma _{t} (f(M)).$
 Therefore $G(\gamma _{t}(M^{10}))=\gamma _{t}(M^{9})f(\gamma _{t}(M))=\gamma _{t}(M^{9})\gamma _{t}(f(M))=\gamma _{t}(G(M^{10}))$.  Consequently $G(L(M^{10}))=L(G(M^{10}))$, as required.  

By the definition of $G$, for the same $M$, and for any $t\in F$, and any number $k$, 
\[G(\gamma _{t}(M^{10k}))=G(\gamma _{t}(M^{10}))^{k}=\gamma _{t}(G(M^{10}))^{k}=\gamma _{t} (G(M^{10k})).\]
  Therefore, $G(L(M^{10k})=L(G(M^{10k}))$. 
The result now follows when we take $M=M_{i}^{m\over n}$ and $k={m_{i}\over 10m}$ and substitute in the above equation.
\end{proof}

\begin{lemma}\label{pierwszy}
  Let $M$ be a finite matrix with entries in $A(j)\cap R$ for some $j$. For almost all $n$ the dimension of the space $R\cap S(L(M^{n}))=\bar {S}(M^{n})$ is smaller than $\sqrt n$. Notice also that since the entries of $M$ are taken from $R$ which is $\gamma _{t}$ invariant for all $t\in F$ then the space $S(L(M^{n}))$ is finite dimensional for every $n$.
\end{lemma}
\begin{proof}
  Since all entries of $M$ are in $R$ then $S(L(M^{n}))=S(M^{n}).$ 
Let $M$ be an $m$ by $m$ matrix, then the dimension of $\bar {S}(M^{n})$ is at most $m^{2}$, which for sufficiently large $n$ is smaller than $\sqrt n $. 
\end{proof}



\section{Supporting lemmas}

 Let $n, m_{1}$ be natural numbers such that $20n$ divides $m_{1}$. Let $M_{1}$ be a matrix with entries in $A^{*}(n)$.
 Let $f$ be a mapping satysfying properties (1)--(5) from Lemma \ref{one} for $m=2m_{1}$ and for the space $S=S(L(M_{1}^{{m_{1}\over n}}))A^{*}(m_{1}).$
 We can then define  mapping $G: A^{*}(j\cdot 10\cdot 2m_{1})\rightarrow A^{*}(j\cdot 10\cdot 2m_{1})$ as in the previous section (for every $j$).

 
In the next three lemmas we will use the following notation. 
 Let $m$ be a natural number and let $V\subseteq A^{*}(m)$ be a linear space.
Denote  
\[E(V, m)=\sum_{j=1}^{\infty } A^{*}(j\cdot 20m-2m)VA^{*}(m){A^{*}}^{1}.\]

We begin with the following lemmas:

\begin{lemma}\label{ker} 
  Let $j, n, m_{1}$ be natural numbers such that $20n$ divides $m_{1}$. Let $M_{1}$ be a matrix with coefficients in $A^{*}(n)$.
 Let $G$ be defined as  at the beginning of this section.
 Then the kernel of $G:A^{*}(j\cdot 20m_{1})\rightarrow A^{*}(j\cdot 20m_{1})$ is equal to $E(V, m_{1})\cap A^{*}(j\cdot 20 m_{1})$, where $V=S(L(M_{1}^{m_{1}\over n}))$.
\end{lemma}
\begin{proof}  By assertion (5) from Lemma \ref{one} applied for $m=2m_{1}$ and $S=VA^{*}(m_{1}),$ there is 
 a linear space $E\subseteq A^{*}(2m_{1})$ and such that 
 $f(r)=r$ for $r\in E$ and $E\oplus S=A^{*}(2m_{1}).$ For $i=1,2, \ldots , j$, let $D=\prod_{i=1}^{j} A^{*}( 18m_{1})E$ and let \[T_{i}=A^{*}(i\cdot 20m_{1}-2m_{1})\cdot V\cdot A^{*}((j-i)\cdot 20m_{1}+m_{1}).\] 
 Observe that $D+\sum _{i=1}^{j} T_{i}=A^{*}(j\cdot 20m_{1})$.
If $r\in T_{i}$ for some $i$, then $G(r)=0$ by the definition of $G$.
 Observe that  $E(V, m_{1})\cap A^{*}(j\cdot 20 m_{1})=\sum _{i=1}^{j}T_{i}$, hence  $E(V, m_{1})\cap A^{*}(j\cdot 20 m_{1})$ is contained in the kernel of $G$.

Let $r$ be in the kernel of $f$. Write $r=t+d$, where $t\in \sum _{i=1}^{j}T_{i}$ and $d\in D$, then 
 by the definition of $G$,
$G(r)=G(t+d)=G(t)+G(d)=G(d)=d$. Recall that $r$ is in the kernel of $f$, so $d=0$ and hence $r\in \sum _{i=1}^{j}T_{i}.$
 It follows that the kernel of $G$ equals $E(V, m_{1})\cap A^{*}(j\cdot 20 m_{1})$.
\end{proof}
 
\begin{lemma}\label{pomocniczy}  
 Let $n, k$ be natural numbers with $n$ even,  and let  $m_{1}<m_{2}<\ldots<m_{k}$ be such that $20n m_{i}$ divides $m_{i+1}$ for all $1\leq i < k$ and $20n$ divides  $m_{1}$. Let $M_{i}$ be matrices with entries in $A^{*}( n)$ and let $G: A^{*}(j\cdot 20m_{1})\rightarrow A^{*}(j\cdot 20m_{1})$ be defined as at the beginning of this section, and let $j={m_{k}\over m_{1}}$. 
 Let $u\in A^{*}(20m_{k})$,  and denote $V_{i}=L(S(M_{i}^{m_{i}/n}))$.  
 Then $u\in \sum_{i=1}^{k }E(V_{i}, m_{i})$ if and only if  
   $G(u)\in \sum_{i=2}^{k }G(E(V_{i}, m_{i})).$
\end{lemma}
\begin{proof}  Suppose that $G(u)\in \sum_{i=2}^{k }G(E(V_{i}, m_{i})).$ It follows that  
$G(u-e)=0$ for some $e\in \sum_{i=2}^{k }E(V_{i}, m_{i})$. Consequently 
  $u-e\in \ker (G)$, and since by Lemma \ref{ker} $\ker (G)\subseteq E(V_{1}, m_{1})$, it follows that 
   $u\in \sum_{i=1}^{k }E(V_{i}, m_{i}).$

 To see the second implication, suppose now that $u\in \sum_{i=1}^{k}E(V_{i}, m_{i});$ then $G(u)\in \sum_{i=1}^{k}G(E(V_{i}, m_{i})).$  By Lemma \ref{ker}, $E(V_{1}, m_{1})\cap 
 A^{*}(20m_{k})\subseteq \ker (G)$, hence $G(L(V_{1}, m_{1}))=0$. It follows that  $G(u)\in \sum_{i=2}^{k }G(E(V_{i}, m_{i}))$, since all the considered spaces are homogeneous. 
\end{proof}
\begin{lemma}\label{helping}  Let notation be as in Lemma \ref{pomocniczy}.
For $i=2,3, \ldots $ denote \[M'_{i}=G(M_{i}^{20m_{1}\over n}), W_{i}=L(S({M'_{i}}^{m_{i}\over 20m_{1}})), T_{i}=G(E(V_{i}, m_{i})\cap A^{*}(20m_{k})).\]  Then for $i\geq 2$ we have 
$T_{i}\subseteq E(W_{i}, m_{i})\cap A^{*}(20m_{k}) \subseteq T_{i}+E(V_{1}, m_{1}).$
\end{lemma}
\begin{proof} Recall that $20m_{1}$ divides $m_{i}$ for all $i\geq 2$; hence 
\[T_{i}=A^{*}(20m_{k})\cap \sum_{j=1}^{\infty }G(A^{*}(j20m_{i}-2m_{i}))G(L(S({M_{i}}^{m_{i}\over n})))G(A^{*}(m_{i}))G({A^{*}}^{1}).\] 
 
Observe that \[{M'_{i}}^{m_{i}\over 20m_{1}}=[G(M_{i}^{20m_{1}\over n})]^{m_{i}\over 20m_{1}}=G({M_{i}}^{m_{i}\over n})\]
 by the definition of mapping $G$. Therefore, and by Lemma \ref{G5} applied for $m=2m_{1}$,  
\[L(S({M'_{i}}^{m_{i}\over 20m_{1}}))=L(S(G({M_{i}}^{m_{i}\over n})))=G(L(S({M_{i}}^{m_{i}\over n}))).\] 
 It follows that 
\[T_{i}=A^{*}(20m_{k})\cap \sum_{j=1}^{\infty }G(A^{*}(j20m_{i}-2m_{i}))(L(S({M'_{i}}^{m_{i}\over 20m_{1}})))G(A^{*}(m_{i}))G({A^{*}}^{1}).\]\label{3}
 It follows that $T_{i}\subseteq E(W_{i}, m_{i})\cap A^{*}(20m_{k})$.

 We will now show that $E(W_{i}, m_{i})\cap A^{*}(20m_{k}) \subseteq T_{i}+E(V_{1}, m_{1}).$
  Recall that the mapping $G$ can be defined on $A^{*}(j\cdot 20m_{1})$ for any $j$, and that $20m_{1}$ divides $m_{i}$ for each $i> 1.$ 
 Observe now that by assertion $(5)$ from Lemma \ref{one}, $\ker (f)+ $Im $(f)=A^{*}(m)$. 
 Therefore, by the construction of mapping $G$ we get that $A^{*}(j\cdot 20m_{1})=($Im $(G)+\ker G)\cap A^{*}(j\cdot 20m_{1})$ for every $j$.
 
It follows that 
$A^{*}(m_{i})\subseteq G(A^{*}(m_{i}))+ \ker (G)\cap A^{*}(m_{i})$.
 By Lemma \ref{ker}, $A^{*}(m_{i})\subseteq G(A^{*}(m_{i}))+  E(V_{1}, m_{1})\cap A^{*}(m_{i})$.
 Similarly, $A^{*}(j\cdot 20m_{i}-2m_{i}))\subseteq G(A^{*}(j20m_{i}-2m_{i}))+E(V_{1}, n_{1})\cap A^{*}(j\cdot 20m_{i}-2m_{i})$.
 It follows that $E(W_{i}, m_{i})\cap A^{*}(20m_{k}) \subseteq T_{i}+E(V_{1}, m_{1}).$ 
\end{proof}
\begin{theorem}\label{rowne} Let notation be as in Lemma \ref{helping}. Let $u\in A^{*}(20m_{k})$. 
 Then  \[u\in \sum_{i=1}^{k }E(V_{i}, m_{i}),\] if and only if  
   \[G(u)\in \sum_{i=2}^{k }E(W_{i}, m_{i}).\]
\end{theorem}
\begin{proof} We will first prove that $G(u)\in \sum_{i=2}^{k }E(W_{i}, m_{i})$ implies $u\in \sum_{i=1}^{k }E(V_{i}, m_{i})$. Assume on the contrary that   $G(u)\in \sum_{i=2}^{k }E(W_{i}, m_{i})$ and $u\notin \sum_{i=1}^{k }E(V_{i}, m_{i})$. Observe that by Lemma \ref{helping},  
  $E(W_{i}, m_{i})\cap A^{*}(20m_{k}) \subseteq \sum_{i=2}^{k}T_{i}+E(V_{1}, m_{1}),$  hence $G(u)\in  \sum_{i=2}^{k}T_{i}+E(V_{1}, m_{1}).$  
 Therefore there is $g\in \sum_{i=2}^{k}T_{i}$ and $h\in E(V_{1}, m_{1})$ and $G(u)=g+h$.
 By assertion (5) from Lemma \ref{one} we get that Im $(G)\cap \ker (G)=0$. Since $G(u)$ and $g$ are in Im$(G)$ and $h\in E(V_{1}, m_{1})$ is in the kernel of mapping $G$, 
then $G(u)-g=h$ implies $G(u)-g=0$. Therefore 
$G(u)\in\sum_{i=2}^{k }T_{i}$, a contradiction with Lemma \ref{pomocniczy}. 

We will now prove that $G(u)\notin \sum_{i=2}^{k }E(W_{i}, m_{i})$ implies $u\notin \sum_{i=1}^{k }E(V_{i}, m_{i})$. 
Suppose that $G(u)\notin \sum_{i=2}^{k }E(W_{i}, m_{i}).$ By Lemma \ref{helping}, $G(u)\notin \sum_{i=2}^{k }T_{i},$ where $T_{i}=G(E(V_{i}, m_{i})\cap A(20m_{k}))$. 
 Observe that by Lemma \ref{ker}, $G(E(V_{1}, m_{1}))=0$. 
Since $G$ is a linear mapping it implies $u\notin  \sum_{i=1}^{k }E(V_{i}, m_{i})$
(as otherwise we would have $G(u)\in \sum_{i=2}^{k }T_{i}.$) 
\end{proof}


\section{Assumptions $1$ and $2$ }

 Let notation be as in Section $5$. Recall that, for a matrix $M$, $S(M)$ denotes the linear space spanned by all entries of $M$, and $L(M)=\sum _{t\in F}\gamma _{t}(M).$ 
 Recall that by $\langle x\rangle $ we denote the ideal generated by $x$ in $A$.

The following statement will be called Assumption $1$ (for $F$-algebra $A$).

{\bf  Assumption 1.} Let $M$ be a matrix with entries in $A^{*}(j)$ for some $j$, and such that for almost all $\alpha $ matrix 
$M^{\alpha }$ has all entries in $\langle x\rangle $. Then there are infinitely many $n$, such that the dimension of the space $R\cap S(L(M^{n}))$ does not exceed $\sqrt {n}$.

 {\em Comment.} Notice that, since $A'$ is a graded algebra,  it is necessary to assume that all elements of the  matrix $M$ have the same degree, since otherwise 
 the entries of $M^{\alpha }$ would have many  homogeneous components and $\dim _{F}(R\cap S(L(M^{n})))$ could exceed $\sqrt {n}$.
  On the other hand, it would be possible to assume that  $M$ is a matrix with entries in $A(j)$ for some $j$; however, for our purpose it suffices to assume that $M$ is a matrix with entries in $A^{*}(j),$ for some $j$.
\begin{remark}\label{quote}
Suppose that Assumption $1$ holds, and let $k$ be a natural number. We can apply Assumption $1$ to matrix $M^{k}$ 
 to get the following implication of Assumption $1$:
 Let $m$ be a natural number. There 
are infinitely many $n$ divisible by $k$ such that the dimension of the 
space $R\cap S(L(M^{n}))$ is less than $\sqrt {n}.$
\end{remark}

\begin{delfin} Let $l, t, m, n$ be natural numbers and let $r_{0}, r_{1},\ldots , r_{t}\in A^{*}$. 
 We define \[e(n,m)(r_{0}, r_{1}, \ldots , r_{t})= \sum_{ i_{1}+\ldots +i_{m}=n}r_{i_{1}}r_{i_{2}}\ldots r_{i_{m}} .\]
\end{delfin}

The following lemma is similar to Lemma $7$ (b) in \cite{nil}. 

\begin{lemma}\label{known} Let $m, n, t, l$ be natural numbers, and let $r_{0}, r_{1},\ldots , r_{l}\in A^{*}(t)$. Denote $e(n,m)=e(n,m)(r_{0}, r_{1}, \ldots , r_{l})$; then for every $0<k<m$
\[e(n, m)=\sum_{i=0}^{n}e(i,k)e(n-i,m-k).\]
\end{lemma}
\begin{proof} It is easier to prove a more general result where 
 $r_{i}$ are free generators of a free algebra and we assign gradation $i$ to element $r_{i}$.
 Such a result can be proved for example by induction on $k$.
 The special case when $r_{i}\in A^{*}(t)$ implies Lemma \ref{known}. 
\end{proof}

\begin{lemma}\label{e}
 Let $l, t, m$ be natural numbers and let $r_{0}, r_{1},\ldots , r_{l}\in A^{*}(t)$.
 Denote $r'_{i}=e(i, m)(r_{0}, r_{1}, \ldots , r_{l})$.
 Then for every $n$ and every $i\leq lmn$,  
\[e(i,n)(r'_{0}, r'_{1}, \ldots , r'_{lm})=e(i,mn)(r_{0}, r_{1}, \ldots ,r_{l}).\]
\end{lemma}
\begin{proof}
We will use induction on $n$. If $n=1$ then the result is clear. Suppose that $n>1$ and that the result holds for all numbers smaller than $n$.
  By Lemma \ref{known}, $e(i,mn)(r_{0}, r_{1}, \ldots ,r_{l})=\sum_{j=0}^{i}e(j,m(n-1))(r_{0}, r_{1}, \ldots ,r_{l})\cdot e(j,m)(r_{0}, \ldots ,r_{l})$. By the inductive assumption  and by Lemma \ref{known}, we get 
$e(i,mn)(r_{0}, r_{1}, \ldots ,r_{l})=\sum_{j=0}^{i}e(j,n-1)(r'_{0}, \ldots , r'_{lm})\cdot e(j,1)(r'_{0}, \ldots ,r'_{lm})= e(j,n)(r'_{0}, \ldots , r'_{l}).$ 
\end{proof}

Recall that $\langle x\rangle $ is the ideal of $A'$ generated by $x$.

\begin{lemma}\label{induction}  Let $F$ be a field, and suppose that Assumption $1$ holds for $F$-algebra $A$.
 Let $n, t$ be natural numbers and let $M$ be a matrix with coefficients in $A^{*}(n)$.
 Assume that  either all entries of $M$ are in $R$ or for almost all $q$ entries of $M^{q}$ are in $\langle x\rangle $.  
 Let $t\geq 1$,  $r_{0}, r_{1}, \ldots , r_{t}\in A^{*}(n)\cap R$, and denote $e(j, k)=e(j, k)(r_{0}, r_{1}, \ldots , r_{t})$.  Assume moreover that there are $r_{i},r_{i'}\neq 0$ such that $r_{i'}\notin F\cdot r_{i}$ for some $0\leq i,i'\leq t$. 
 Then there exist $m$ and $j, j'$ such that $20n$ divides $m$ and 
\[e(j,{20m\over n})\notin A^{*}(18m)VA^{*}(m)\] and 
\[e(j',{20m\over n})\notin F \cdot e(j,{20m\over n})+A^{*}(18m)VA^{*}(m),\]
 where $V=L(S(M^{m\over n}))\subseteq A^{*}(m).$ 
 Moreover, if $c>0$ is a natural number then we can assume that $20nc$ divides $m$.
\end{lemma}
 \begin{proof} Consider elements $e(i,{20m\over n})$ for all $m$ divisible by $n$.  Observe that, for almost all such $m$,  the linear space spanned by $e(i,{20m\over n})$ for $i=0,1,2, \ldots $ has dimension larger than $\sqrt {20m\over n}+2$. By Assumption $1$ and Lemma \ref{pierwszy},
$\dim_{F}V\leq \sqrt {20m\over n}$  for infinitely many $m$ (moreover by Remark \ref{quote} we can assume that infinitely many $m$ are divisible by $20nc$). Consequently 
 there is $m$ divisible by $n$ and $j,j'$ such that  
$e(j,{m\over n})\notin V$ and 
$e(j',{m\over n})\notin V +F \cdot e(j', {m\over n}).$ Moreover, we can assume that $20nc$ divides $m$, by the Remark \ref{quote}.

 Let $j$ be minimal such that $e(j, {m\over n})\notin V$ and $j'$ minimal such that $e(j', m)\notin V+F\cdot e(j,{m\over n}).$
 We claim that $e(j,{20m\over n})\notin A^{*}(18m)VA(m)$ and $e(j', {20m\over n})\notin A^{*}(18m)VA(m) +F\cdot e(j,{20m\over n}).$

  Notice that $e(0,{m\over n})={r_{0}}^{{m\over n}}$. Recall that $A^{*}$ is a graded algebra, and therefore $e(j, {m\over n})e(0,{m\over n})\notin VA^{*}(m_{1})$; this can be seen by comparing the elements  from $A^{*}(m)$ at the end of each side. 
 By Lemma \ref{known},  $e(j, {2m\over n})=e(j,{m\over n})e(0,{m\over n})+\sum_{i<j}e(i,{m\over n})e(j-i,{m\over n})$.  
 By the minimality of $j$, we get $e(j, {2m\over n})\notin VA^{*}(m)$. Notice also that, by a similar argument $j$ is minimal such that $e(j, {2m\over n})\notin VA^{*}(m)$. Observe now that $e(0, {18m\over n})e(j, {2m\over n})\notin A^{*}(18m)VA^{*}(m);$ this can be seen by comparing the elements from $A^{*}(18m)$  
at the beginning of each side. By Lemma \ref{known}, $e(j, {20m\over n})=e(0, {18m\over n})e(j,{2m\over n})+\sum_{i<j}e(j-i,{18m\over n})e(i, {2m\over n}).$  Recall that $j$ was minimal such that $e(j, {2m\over n})\notin VA^{*}(m)$, therefore
 $e(j, {20m\over n})\notin A^{*}(18m)VA(m).$

Observe now that, since $e(j', {m\over n})\notin V+F\cdot e(j, {m\over n}),$ then by the same reasoning as above
 applied to the set $T'=V+F\cdot e(j,{m\over n})$ instead of the set $V$, we get 
\[e(j', {20m\over n})\notin A^{*}(18m)T'A^{*}(m).\]
 By the definition of $T'$ we have $e(j, {m\over n})\in T'$ and by the minimality of $j$, $e(i, {m\over n})\in V$ for $i<j$ . By Lemma \ref{known}, 
\[e(j,{20m\over n})\in A^{*}(18m)\sum _{i=0}^{j}e(i,{m\over n})A^{*}(m)\subseteq A^{*}(18m)T'A^{*}(m).\] Therefore $A^{*}(18m)VA(m)+F\cdot e(j, 20m) \subseteq  A^{*}(18m)T'A^{*}(m)$. 
 Hence $e(j', {20m\over n})\notin A^{*}(18m)VA^{*}(m)+F\cdot e(j, 20m),$ as required.
\end{proof}

We will now introduce  Assumption $2$. We introduce this Assumption to shorten the statements of  several theorems, where we will simply write {\em let Assumption $2$ hold} instead of writing the sentences from below. The Assumption $2$ simply says that numbers $m_{i}$ and matrices $M_{i}$ satisfy some conditions, these conditions are now described.

{\bf Assumption 2.}  
 Let $n, k$ be natural numbers with $n$ even,  and let  $m_{1}<m_{2}<\ldots<m_{k}$ be such that $20n m_{i}$ divides $m_{i+1}$ for all $1\leq i < k$ and $20n$ divides  $m_{1}$. Let $M_{i}$  for $i=1,2, \ldots , k+1$ be matrices with entries in $A^{*}( n)$ and 
 assume that either ${M_{i}}^{q}$ has entries in $\langle x\rangle $ for almost all $q$, or $M_{i}$ has entries in $R$.

We will use the following notation.
The mapping  $G: A^{*}(t\cdot 20m_{1})\rightarrow A^{*}(t\cdot 20m_{1})$ is defined as in Lemma \ref{ker} for $t={m_{k}\over m_{1}}$. 
 
 For $i=2,3, \ldots $ denote  $V_{i}=L(S(M_{i}^{m_{i}/n})), M'_{i}=G(M_{i}^{20m_{1}\over n})$, $W_{i}=L(S({M'_{i}}^{m_{i}\over 20m_{1}}))$ and $T_{i}=G(E(V_{i}, m_{i})\cap A^{*}(20m_{k})).$
\begin{lemma}\label{linearmapping} Suppose that Assumption $1$ holds for $F$-algebra $A$. Suppose that Assumption $2$ holds.  Let $k\geq 0$ be a natural number.
 Then the following conditions hold. 
\begin{itemize}
\item[1.]  If $k=0$, let $r_{0},r_{1},\ldots  , r_{l}\in A^{*}(n)$. 
Suppose that there are $j, j'$ such that \[\alpha \cdot r_{j}+\beta \cdot r_{j'}\neq 0,\] provided that $\alpha ,\beta \in F$ are not both zero. 
 Denote $e(j, k)=e(j, k)(r_{0}, r_{1}, \ldots , r_{t})$.
 Then there exists $m_{1}$ such that  $20n$ divides $m_{1}$ and 
\[\alpha \cdot  e(l,{20m_{1}\over n})+\beta \cdot e(l', {20m_{1}\over n}) \notin E(V_{1}, m_{1}).\]
\item[2.] If $k>0$, let  
 $r_{0}, r_{1}, \ldots , r_{l}\in R\cap A^{*}(20m_{k})$ for some $l\geq 1$. Suppose that there are $j, j'$ such that \[\alpha \cdot r_{j}+\beta \cdot r_{j'}\notin W,\] provided that $\alpha ,\beta \in F$ are not both zero, where $W= \sum_{i=1}^{k}E(V_{i}, m_{i})$.
 Denote $e(j, k)=e(j, k)(r_{0}, r_{1}, \ldots , r_{t})$.

 Then there exist $m_{k+1}$ such that  for some $l, l'$,  
\[\alpha \cdot  e(l,{m_{k+1}\over m_{k}})+\beta \cdot e(l', {m_{k+1}\over m_{k}}) \notin W',\]  provided that $\alpha ,\beta \in F$ are not both zero,
 where $W'= \sum _{i=1}^{k+1}E(V_{i}, m_{i})$ and $V_{k+1}=L(S({M_{k+1}}^{m_{k+1}\over n})).$ Moreover, $20nm_{k}$ divides $m_{k+1}$.
\end{itemize}
\end{lemma}
\begin{proof} We will proceed by induction on $k$. If $k=0$ then the result follows from Lemma \ref{induction}  applied for  $m=m_{1}$ and matrix $M=M_{1}$.
Let $k\geq 1$ and assume that the thesis is true for all numbers  smaller than $k$;  we will prove it for $k$.
 
Suppose that $k=1$. By the assumption there are $j, j'$  such that $\alpha \cdot r_{j}+\beta \cdot r_{j'}\notin E(V_{1}, m_{1})$ provided that $\alpha ,\beta $ are not both zero.   Let $f, G$ be as in Theorems \ref{ker} and \ref{rowne}; then by Theorem \ref{rowne}
\[\alpha \cdot G(r_{j}) +\beta \cdot G(r_{j'})\notin  \sum_{i=2}^{k}E(W_{i}, m_{i})=0,\] since $k=1$.
   Next we apply Lemma \ref{induction} for matrix $M=G(M_{1}^{20m_{1}\over n})$,  for elements $G(r_{i})$ instead of elements $r_{i}$, and  for $n=20m_{1}$, and we find  $m $ such that if $\alpha , \beta \in F$ are not both zero then
 
\[\alpha \bar {e}(j',{20m\over 20m_{1}}) +\beta  \bar {e}(j,{20m\over 20m_{1}})\notin A^{*}(18m)W_{1}A^{*}(m),\]
 where $\bar {e}(i,{20m\over 20m_{1}})=e((i,{20m\over 20m_{1}})(G(r_{1}), \ldots , G(r_{l}))$ and $W_{1}=S(L(G(M^{m\over 20m_{1}})))=S(L(G(M_{1}^{m\over n})))$.

 Observe that $\bar {e}(i,{m\over m_{1}})=G(e(i,{m\over m_{1}})).$
  Let $m_{2}=m$ 
 then 
\[\alpha G(e(j',{m_{2}\over m_{1}})) +\beta  G(e(j,{m_{2}\over m_{1}}))\notin A^{*}(18m_{2})W_{1}A^{*}(m_{2}),\]

 By Theorem \ref{rowne},
\[\alpha e(j',{m_{2}\over m_{1}}) +\beta e(j,{m_{2}\over m_{1}})\notin W',\]
 hence Lemma \ref{linearmapping} holds. 

Suppose that $k>1$ and assume that the result holds for all numbers smaller than $k$. 
 By the assumption there are $j, j'$  such that $\alpha \cdot r_{j}+\beta \cdot r_{j'}\notin W$ provided that $\alpha ,\beta $ are not both zero. 
  Let $f, G$ be as in Theorems \ref{ker} and \ref{rowne}; then by Theorem \ref{rowne}
\[\alpha \cdot G(r_{j}) +\beta \cdot G(r_{j'})\notin \sum_{i=2}^{k}E(W_{i}, m_{i}),\]  provided that $\alpha ,\beta \in F$ are not both zero.

 Recall that $M'_{i}=G(M_{i})^{20m_{1}\over n}$ and $W_{i}=L(S({M'_{i}}^{m_{i}\over 20m_{1}}))$ for $i=2,3, \ldots , k$.
 Observe that the number of matrices $M'_{i}$  is $k-1$, hence we can apply the inductive assumption to  matrices $M'_{2}, M'_{3}, \ldots , M'_{k}$. Namely we enumerate $M''_{i}=M'_{i+1}$, $m'_{i}=m_{i+1}$ for every $i$ and 
  we apply the inductive assumption to matrices $M''_{i}$ for $i\leq k-1$  and to  numbers $m'_{i}$ and to elements $\bar {r}_{i}=G(r_{i})$ and to $n=20m_{1}$. 
 We  obtain that 
 there is $m=m_{k+1}$ and  $j, j'$ such that  if $\alpha , \beta \in F$ are not both zero then 
\[\alpha \cdot G(e(j, {m_{k+1}\over m_{k}}))+\beta \cdot G(e(j',{m_{k+1}\over m_{k}})) \notin \sum_{i=2}^{k}E(W_{i}, m_{i}),\] 
 since $G(e(j, {m_{k+1}\over m_{k}}))=e(j, {m_{k+1}\over m_{k}})(G(r_{0}), G(r_{1}), \ldots , G(r_{l}))$. 
By Theorem \ref{rowne},  
we get $\alpha \cdot e(j, {m_{k+1}\over m_{k}})+\beta \cdot e(j',{m_{k+1}\over m_{k}})) \notin W',$ 
 unless $\alpha =\beta =0$.\end{proof}

 Let $X_{1}, X_{2}, \ldots $ be matrices as in Theorem \ref{Jacobson}.
 Let $Y_{1}, Y_{2}, \ldots $ be matrices  such that $Y_{i}$ has entries in $A^{*}(2)$ and $X_{i}^{2}=Y_{i}+M(B')$, where $M(B)$ is the set of matrices with entries in $M(B')$. Recall that  Let $B'\subseteq A(2)$ is the linear $F$- space spanned by elements $ax^{i}bx^{j}$ and $bx^{i}ax^{j}$ for all $i,j\geq 0$.

Recall that $e(i,n)(a^{2}, b^{2})$ is the sum of all products of  $n$ elements, such that $i$ of them are equal to  $b^{2}$ and $n-i$ of them are equal to $a^{2}$. 

\begin{theorem}\label{important3} Let $F$ be a countable field, and suppose that Assumption $1$ holds for $F$-algebra $A$. For every $i,n$ denote $e(i,n)=e(i,n)(a^{2},b^{2}).$
Let $Y_{1}, Y_{2}, \ldots $ be as above, then there are natural numbers $m_{1}<m_{2}<\ldots $ such that  
  $20n m_{i}$ divides $m_{i+1}$ for all  $i$ and $40$ divides $m_{1}$. Moreover, for every $n$ there are $j,j'$
 such that $\alpha e(j,n)+\beta e(j',n)\notin T$, provided that $\alpha ,\beta \in F$ are not both zero, where $T=\sum_{i=1}^{\infty }E(V_{i}, m_{i})$ and $V_{i}=L(S(Y_{i}^{m_{i}\over 2})).$ 
\end{theorem}
\begin{proof} Notice that $e(j,n)\subseteq A^{*}(2n)$ for every $n$. 
 Observe first that if $u(j,n)\in T$ for some $n$ and all $j\leq n$, then 
 by Lemma \ref{known}  for every $n'>n$ we have $e(j,n')\in T,$ for all $j\leq 2n'$. 
 Therefore, it is sufficient to prove that there are $m_{1}< m_{2}< \ldots $ and $j_{1}, j_{2}, \ldots $ and $j'_{1}, j'_{2}, \ldots $ such that for every $k$,   
 \[\alpha \cdot e(j_{k},10m_{k})+ \beta \cdot e(j' _{k},10m_{k})\notin T,\] provided that $\alpha , \beta \in F$ are not both zero.
  Notice that \[T\cap A^{*}(20m_{k})=\sum_{i=1}^{k}E(V_{i}, m_{i}),\] since all spaces $E(V_{i}, m_{i})$ are homogeneous. 

 We will construct numbers $m_{1}, m_{2}, \ldots $ inductively using Lemma \ref{linearmapping}. By Lemma \ref{linearmapping}, applied for $n=2$ and matrices $M_{i}=Y_{i}$, there is $m_{1}$ and $j, j'$ such that 
$\alpha \cdot e(j, 10m_{1})+\beta \cdot e(j', 10m_{1})\notin T,$ provided that $\alpha, \beta \in F$ are not both zero, moreover 40 divides $m_{1}$. 

Suppose now that for some $k\geq 1$ we constructed elements $m_{1}, \ldots , m_{k}$ such that if $\alpha , \beta \in F$ are not both zero then
 \[\alpha \cdot e(j_{k},10m_{k})+ \beta \cdot e(j' _{k},10m_{k})\notin T.\]
 By Lemma \ref{linearmapping} there are $l, l'$ such that 
\[\alpha \cdot  e(l,{m_{k+1}\over m_{k}})( r_{0}, \ldots , r_{10m_{k}})+\beta \cdot e(l', {m_{k+1}\over m_{k}})( r_{0}, \ldots , r_{10m_{k}}) \notin T,\]
  where $r_{i}=e(i, 10m_{k})$.
  By Lemma \ref{e}, we get 
$e(i,10m_{k+1})=e(i,{m_{k+1}\over m_{k}})( r_{0}, \ldots , r_{10m_{k}}).$
 Consequently, $\alpha \cdot  e(l, 10 m_{k+1})+\beta \cdot e(l', 10m_{k+1}) \notin T,$  therefore we constructed $m_{k+1}$  satisfying the thesis of our theorem. Continuing in this way we construct all elements $m_{i}$.
\end{proof}


\section{ Nility}

Let $A^{*}$ be a subalgebra of $A$ generated by elements $ax^{i}ax^{j}$ and $bx^{i}bx^{j}$ for all $i,j\geq 0$.

 Let $B'\subseteq A(2)$ be a linear $F$- space spanned by elements $ax^{i}bx^{j}$ and $bx^{i}ax^{j}$ for all $i,j\geq 0$.
 Let $B=\sum_{i=0}^{\infty }A(2i)B'A$. Observe that $A=A^{*}+B$ and $A^{*}\cap B=0$.

 In this chapter we  denote $e(i,n)=e(i,n)(a^{2}, b^{2})$ to be the sum of all products of  $n$ elements, such that $i$ of them are equal to  $b^{2}$ and $n-i$ of them are equal to $a^{2}$, so $e(k,n)\in A^{*}(2n).$
 Recall that an ideal in $A$ is homogeneous if it is homogeneous with respect to the gradation given by assigning gradation $1$ to elements $a$ and $b$ and gradation $0$ to element $x$.
\begin{lemma}\label{nil} 
 Suppose that $J$ is a homogeneous  ideal in $A$ such that  $R/R\cap J$ is a nil algebra. 
 Then there is $m>0$ such that $e(k, m)\in J+B$, for every $0\leq k\leq m$. 
\end{lemma}
\begin{proof}
 By assumption, there is a number $m$ such that for every
 $n\geq m$ we have  $(a+b^{2})^{n}\in J$.
Let $v(k,n)$ be the sum of all products of $k$ elements $b^{2}$ and $n-2k$ elements $a$; then $v(k,n)\in A(n).$
 Observe that 
 \[(a+b^{2})^{n}=\sum_{0\leq k\leq {n\over 2}}v(k, n+k)\in J,\] and since $J$ is homogeneous and $v(k, n+k)\in A(n+k)$ it follows that
 \[v(k,n+k)\in J\] for every natural $k\leq {n\over 2},$ and every $n\geq m$.

Therefore $v(k,2m)=v(k, (2m-k)+k)\in J$ for every $0\leq k\leq m$, since $2m-k\geq m$. 
 Observe that $e(k,m)-v(k,2m)\in B$, so $e(k,m)\in B+J$.
\end{proof}
\begin{theorem}\label{Jacobson2} Let notation be as in Theorem \ref{Jacobson}, and denote $Q=\sum_{i=1}^{\infty }S_{i}$. 
  Then there is a homogeneous ideal $J$ in $A$ which is a platinum ideal,   $A/J$ is Jacobson radical and $J$ is contained in $L(Q)$.  Moreover, $JA'\subseteq J$ so $J$ is a right ideal in $A'$.
\end{theorem}
\begin{proof} By Theorem \ref{Jacobson}, there is an ideal $I$ in $A$ such that $I\subseteq Q$. Denote $J=L(I)$; then $L(I)=\sum_{t\in F}\gamma _{t}(I)$. Observe that $I\subseteq Q$ implies $L(I)\subseteq L(Q)$. 
  We claim that $L(I)$ is an ideal in $A$, and a right  ideal in $A'$. We need to show that if $\alpha \in L(Q)$ then  
 $r\alpha , \alpha r'\in L(I)$, for every $r\in A$ and $r'\in A'$.
 Since $\alpha \in L(Q)$ then $\alpha =\sum _{t\in W}\gamma _{t}(s_{t})$ for some finite subset $W$ of $F$, and where $s_{t}\in I$. 
  Observe that $r\gamma _{t}(s)=\gamma_{t}(\gamma _{-t}(r)s_{t})\in \gamma_{-t}(I)\subseteq L(I)$ and 
$\gamma _{t}(s)r'=\gamma _{t}(s\gamma _{-t}(r'))\subseteq \gamma _{t}(I)\subseteq L(I)$, since by Theorem \ref{Jacobson}, $I$ is a right ideal in $A'.$
 By  Lemma \ref{jeden}, $L(I)$ is a platinum ideal in $A$. 
 Since $A/I$ is a Jacobson radical then $A/L(I)$ is a Jacobson radical, so we can set $J=L(I)$.
\end{proof}


\section{Assumption $1$ implies Theorem \ref{main}}

 The aim of this section is to prove the following.
\begin{theorem}\label{main2} Let $F$ be a field, and suppose that Assumption $1$ holds for $F$-algebra $A$.
 Then there is an $F$-algebra $Z$ and a derivation $D$ on $Z$ such that the differential polynomial ring $Z[y; D]$ is Jacobson radical but $Z$ is not nil.
\end{theorem}
\begin{proof} 
 Let $m_{1}, m_{2}, \ldots $ be as in Theorem \ref{important3} and denote  $T=\sum_{i=1}^{\infty }E(V_{i}, m_{i})$ and $V_{i}=L(S(Y_{i}^{m_{i}\over 2})).$ 
 By Theorem \ref{important3} for every $n$ there are $j, j'$
 such that $\alpha \cdot e(j,n)+\beta \cdot e(j', n)\notin T$, provided that $\alpha ,\beta \in F$ are not both zero.
 Observe that since $e(j,n)\in A^{*}$ and $T\subseteq A^{*}$ and $A^{*}\cap B=0$, it follows that 
for every $n$ there are $j, j'\leq n$
 such that \[\alpha \cdot e(j,n)+\beta \cdot e(j', n)\notin T+B,\] provided that $\alpha ,\beta \in F$ are not both zero.

 By Theorem \ref{Jacobson2} applied for such $m_{1}, m_{2}, \ldots $ we get that  
  there is a homogeneous ideal $J$ in $A$ which is a platinum ideal,   $A/J$ is Jacobson radical and $J$ is contained in $L(Q)$, where $Q=\sum_{i=1}^{\infty }S_{i}$  as in Theorem \ref{Jacobson}.  Moreover, $JA'\subseteq J$ so $J$ is a right ideal in $A'$.
 Denote $J^{(*)}=J+xJ+x^{2}J+\ldots =\sum_{i=0}^{\infty }x^{i}J$, then $J^{(*)}$ is a platinum ideal in $A'$. 
 Let $ A^{(*)}=A+xA+x^{2}A+\ldots $; then $A^{(*)}$ is an $F$-algebra and $J^{(*)}$ is an ideal in $A^{(*)}$. By Lemma \ref{mily}, $ A^{(*)}/J^{(*)}$ is Jacobson radical.
  In addition, if $r+J$ is not a nilpotent in $A/J$ for some $r\in R$, then $r+J^{(*)}$ is not a nilpotent in $ A^{(*)}/J^{(*)}$. 
 
 We will now show that $R/{R\cap J}$ is not a nil algebra.
 Observe now that $L(S_{i})\subseteq \sum_{j=1}^{\infty}A(j\cdot 20m_{i}-2m_{i})L(S'_{i})A(m_{i})A^{1}$, where $S'_{i}=S({X_{i}}^{m_{i}})$ (since $L(A(j))\subseteq A(j)$ for any $j$), and that  $X_{i}^{2}=Y_{i}+B_{i}$, where $B_{i}$ are matrices with entries in $B$. Moreover $A(j)\subseteq A^{*}(j)+B$ for every even $j$, by the definition of $B$.
 It follows that $L(S_{i}A^{1})\subseteq E(V_{i}, m_{i})+B$ for every $i$. By the first part of this proof we get that 
 for every $n$ there are $j, j'\leq n$ such that $\alpha \cdot e(j,n)+\beta \cdot e(j', n)\notin T+B$. On the other hand $ J\subseteq  L(Q)=\sum_{i=1}^{\infty }L(S_{i})\subseteq T+B$. 
It follows that \[\alpha \cdot e(j,n)+\beta \cdot e(j', n)\notin J\]  provided that $\alpha ,\beta \in F$ are not both zero. By Lemma \ref{nil} we get that $R/R\cap J$ is not a nil algebra. So there is $r\in R$ such that $r+J$ is not nilpotent in $A/J$. By Lemma \ref{mily}, $r+J^{(*)}$ is not nilpotent in $A^{(*)}/J^{(*)}$.

 Let $P$ be as in  Theorem \ref{useful2}; so if $a\in P$ then $ax-xa\in P$, $R\subseteq P$  and $P$ is the smallest subring of $A'$ with this property.
 We can apply Theorem \ref{useful2} to $I=J^{(*)}$ and let $Z=P/\bar {J}$ where $\bar J=P\cap J^{(*)}$. Then $Z[y; D]$ is a differential polynomial ring with $y(r+ \bar {J})-(r+ \bar {J})y=D(r)+\bar {J}$, where $D(r+\bar {J})=xr-rx+\bar {J}$.
 By Theorem \ref{useful2}, $Z[y;D]$ can be embedded  in $A'/J^{(*)}$ and 
 the image of $Z[y;D]$ equals $A^{(*)}/J^{(*)}$ where $A^{(*)}=A+xA+x^{2}A+\ldots $, hence the image of $Z[y; D]$ is Jacobson radical, and so $Z[y; D]$ is Jacobson radical.  
 
It remains to show that $Z$ is not a nil ring.  By Lemma \ref{nil} we get that $R/R\cap J$ is not a nil algebra. So there is $r\in R$ such that $r+J$ is not nilpotent in $A/J$.  By Lemma \ref{mily} we get  that $r^{n}\notin J^{(*)}$, and therefore $r^{n}\notin \bar {J}=J^{(*)}\cap P$ for $n\geq 1$. Therefore $r+\bar {J}$ is not nilpotent in $P/ \bar {J}$, so $P/\bar {J}$ is not nil. 
\end{proof}

\section{Assumption $3$} 

 Let $C$ be a matrix with entries $a_{i,j}\in F$ and let $r\in A$ then $Ca$ will denote the matrix with entries $a_{i,j}\cdot r$. 

Let  $M=\sum_{i=1}^{\xi }A_{i}a_{i}$, where for each $i$, $A_{i}$ is a matrix with coefficients in $F$, and $a_{1}, a_{2}, \ldots , a_{\xi }\in A$. 
  By $H$ we will denote the $F$-algebra generated by matrices $A_{i}$, and by $W$ we will denote the Wedderburn radical of $H$, which is  the sum of all nilpotent ideals in $H$. Notice that $W$ is the largest nilpotent ideal in $H$, since $H$ is finite dimensional. By $s$ we will denote the smallest  natural number such that  $W^{s}=0$.

 We will say that a matrix $M$ satisfies {\bf Assumption 3} if the following holds: 
 \begin{itemize}
\item [1. ]   $M=\sum_{i=1}^{\xi }A_{i}a_{i}$, where for each $i$, $A_{i}$ is a matrix with coefficients in $F$, and $a_{1}, a_{2}, \ldots , a_{\xi }$ are elements from $A$ which are linearly independent over $F$ and have the same degree -- say $\alpha $.  
Moreover,  $a_{i}\in A(\alpha )\cap \langle x\rangle $ for each $i\leq \xi $  (where $\langle x\rangle $ is the ideal generated by $x$ in $A'$).
 \item [2.] $A_{1}=e$, where $e$ is a matrix such that $e^{2}=e$,  $r-er\in W$,  $r-re\in W$.  where $W$ is as at the beginning of this section. 
\end{itemize}

 Let $A'^{1}[y]$ be the polynomial ring in variable $y$ over ring $A'^{1}$.

\begin{delfin}\label{w(m,n)} For a variable $y$ (commuting with all elements from $A'$) let $\gamma_{y }:A'\rightarrow A'^{1}[y]$ be a homomorphism of  algebras such that $\gamma_{y}(a)=a, \gamma_{y}(b)=b$ and $\gamma_{y }(x)=x+y $.
 
 For a variable $y$ (commuting with all elements from $A'$) we can write \[\gamma _{y}(M)=\sum_{i=0}^{t}y^{i}M_i,\]  where all entries of $M_{i}$ are homogeneous elements of $A$ of degree $\alpha $, similarly to the entries of matrix $M$. 

Let matrix  $M=\sum_{i=1}^{\xi }A_{i}a_{i}$ satisfy Assumption $3$. Let $\gamma _{y}(M)=\sum_{i=0}^{t}y^{i}M_i.$  Denote 

\[w(n,m)=\sum_{i_1+i_2+\ldots +i_m=n}M_{i_1}M_{i_2}\ldots M_{i_m}.\]
\end{delfin}

 The following lemma is known, and is another variant of  Lemma $7$ in \cite{nil}.
\begin{lemma}\label{known3}
 Let notation be as above. 
 Let $m, n$ be natural numbers; then for every $k<m$
\[w(n, m)=\sum_{i=0}^{n}w(i,k)w(n-i,m-k).\]
\end{lemma}

\begin{lemma}\label{basis} Let $F$ be an infinite field. Let matrix  $M=\sum_{i=1}^{\xi }A_{i}a_{i}$ satisfy Assumption $3$ and $\gamma _{y}(M)=\sum_{i=0}^{t}y^{i}M_i$. Then for each $j$, 
\[ M_{j}\subseteq \sum_{i=1}^{\xi }A_{i}A.\]  
Moreover, for every $n,m, w(m,n)\subseteq HA$ where $H$ is the algebra generated by matices $A_{1}, \ldots ,A_{\xi }$. 
\end{lemma}
\begin{proof}
 Recall that $\gamma _{y}(M)=\sum_{i=0}^{t}y^{i}M_i$; therefore by substituting $y=\alpha _{i}$ for  $\alpha  \in F$, we get 
  $\gamma _{\alpha  }(M)=\sum_{i=0}^{t}\alpha ^{i}M_i$. By applying this for $\alpha =t_{1}, \ldots , \alpha =t_{\xi }$ for various $t_{1}, \ldots , t_{\xi}\in F$, and using the fact that a Vandermonde matrix  is invertible that
  each $M_{i}$ is a linear combination of matrices $\gamma _{\alpha }(M)$ for various $\alpha \in F$.  Recall that $M=\sum_{i=1}^{\xi }A_{i}a_{i}$, hence $\gamma _{\alpha }(M)=\sum_{i=1}^{\xi }A_{i}\gamma _{\alpha }(a_{i}).$ Therefore each $M_{j}$ is a linear combination of matrices $A_{i}\gamma _{\alpha }(a_{i}) \subseteq A_{i}A$. By the definition  of elements $w(n,m)$ (Definition \ref{w(m,n)}) we get $w(n,m)\in HA$. \end{proof}

\begin{lemma} \label{wl}  Let $F$ be an infinite field and let $n$ be a natural number. 
 Let $Q$ be the linear  subspace of $A$  spanned by all entries of matrices  $w(n,m)$ for $n=0, 1,  \ldots $;  then  all entries of the matrix $M^{m}$ belong to $Q$.
 Moreover $Q=S(L(M^{m}))$.  
\end{lemma}
\begin{proof} By assumption $\gamma _{y} (M^{m})=(\sum_{i=0}^{t}y^{i}M_i)^{m},$ hence for any $\alpha \in F$, when substituting $y=\alpha $ we get 
\[\gamma _{\alpha }(M^{m})=(\sum_{i=0}^{t}\alpha ^{i}M_i)^{m}=\sum _{i=0}^{ m\cdot t}\alpha ^{i}w(i,m).\] Therefore every entry of a matrix 
$\gamma _{\alpha }(M^{m})$ is in the subspace generated by entries of matrices $w(i,n)$ for various $i$; hence $S(L(M^{m})\subseteq Q$. 

On the other hand,  let $\alpha _{1}, \alpha _{2}, \ldots , \alpha _{tm}$  be non-zero distinct elements of $F,$
 then we can substitute $\alpha =\alpha _{j}$ into the equation
\[ \sum _{i=0}^{t\cdot m}\alpha ^{i}\cdot w(i,m)=\gamma _{\alpha }(M^{m}).\]
 We can then use the Vandermonde matrix argument to show that 
 each $w(i,m)$ is a linear combination of elements  $\gamma _{\alpha _{i} }(M^{m})$ for $1\leq i\leq tm$.
 Therefore  $Q\subseteq  S(L(M^{m}))$.
\end{proof}



\section{Embedding platinum subspaces in bigger subspaces}

 Let notation be as in the previous section.  Let $M$ be a matrix satisfying Assumption $3.$ Let  $H$, $W$, $e$ be as in Assumption $3$.

\begin{delfin}
Recall that $H$ can be considered as a subalgebra of some matrix algebra over $F$. Let $H'$ denote algebra $H+FI$, where $I$ is the identity matrix. 

 Recall that \[w(n,m)=\sum_{i_1+ i_2+\ldots +i_m=n}M_{i_1}M_{i_2}M_{i_{3}}\ldots M_{i_m}.\]
Let $s$ be such that $W^{s}=0$.
 
Let $f_1, f_2, \ldots $ be elements from $H'.$  We define element $(f_{1}, f_{2}, \ldots) *w(m,n)$ for $j=1,2, \ldots $ in the following way:

\[(f_{1}, f_{2}, \ldots )*w(n,m)=\sum_{i_1+ i_2+\ldots +i_m=n}f_{1}M_{i_1}f_{2}M_{i_2}f_{3}M_{i_3}\ldots f_{m}M_{i_m}.\]
\end{delfin}
\begin{lemma}\label{zerowy} Let notation be as above and let $m$ be a natural number larger than $s$ (where $W^{s}=0$).
 Let $u=(f_1, f_2, \ldots )$ be such that $f_1, f_2, \ldots , f_{s}\in W\cup \{I-e\}$;  then $u*w(n,m)=0$, for every $n$.
\end{lemma}
\begin{proof} Observe that for every $k$, $f_{k}M_{i_{k}}$ is a matrix with entries in $W;$ hence $u*w(n,m)$ has entries in 
 $W^{s}=0$.
\end{proof}
\begin{delfin}\label{elements}  For $n\geq 1$ define $t_{n}=(f_1, f_2, \ldots )$, where $f_1=f_2=\ldots =f_{n-1}=I-e$, $f_n=e$ and $I=f_{n+1}=f_{n+2}=\ldots .$
 For $n\geq 0$ define $t'_{n}=(f_1, f_2, \ldots )$, where $f_1=f_2=\ldots = f_{n}=I-e$ and 
 $I=f_{n+1}=f_{n+2}=\ldots .$
\end{delfin}
\begin{lemma}\label{inclu}  Let notation be as above. Then for every $m,n$, 
\[ w(n, m)=\sum _{i=0}^{s}t_{i}*w(n,m).\]
\end{lemma}
\begin{proof} 
 Observe first that $w(n,m)=(I, I, \ldots )* w(n,m)=t_{1}*w(n,m)+t'_{1}*w(n,m)$;
 then $t'_{1}*w(n,m)=t_{2}*w(n,m)+t'_{2}*w(n,m)$, and for every $i$, 
$t'_{i}*w(n,m)=t_{i+1}*w(n,m)+t'_{i+1}*w(n,m)$. By summing all these equations for $i=1,2, \ldots , s$, and the equation 
 $w(n,m)=t_{1}*w(n,m)+t'_{1}*w(n,m)$, 
 we get $w(n,m)+\sum_{i=1}^{s}t'_{i}*w(n,m)=\sum_{i=1}^{s+1}t_{i}*w(n,m)+t'_{i}*w(n,m).$ By Lemma \ref{zerowy}, $t_{s+1}=t'_{s+1}=t_{s}=0.$
 It follows that  
 $w(n,m)=\sum _{i=0}^{s}t_{i}*w(n,m)$. 
\end{proof}

\begin{delfin} (Definition of set $D(M)$) 

 Let matrix $M$ satisfy Assumption $3$, and let $H$, $W$, $e$ be as in Assumption $3$. 

Fix  $E_{1}, E_{2}, \ldots , E_{\beta }$ -  a basis of $W$ for some $\beta $ (recall that $W$  is the Wedderburn radical of $H$).

Let $E=\{I-e, E_{1}, E_{2}, \ldots , E_{\beta }\}.$ 

 Let $ P_{1}, P_{2}, \ldots , P_{\beta '}$ be such that $E_{1}, \ldots , E_{\beta }, P_{1}, \ldots , P_{\beta '}$ span algebra $H$. 
   We can assume that for every $i$  \[P_{i}e=P_{i}.\] It follows because $H=He+H(I-e)\subseteq He+W$. Moreover $e=he$ for $h=e$. 

 Recall that $W^{s}=0$. 

Let $0\leq q<s$. We say that element $(f_{1}, f_{2} , \ldots )$ is good and has distance $q+1$  if   
 $f_{1}, \ldots , f_{q}\in E,$ $ f_{q+1}\in \{e, P_{1}, P_{2}, \ldots ,P_{\beta '}\}$ and $ I=f_{q+2}= f_{q+3}=\ldots .$

 The set of all good elements will be denoted $D(M)$. 
\end{delfin}
\begin{lemma} \label{finite}
Let $M$ be a matrix satisfying Assumption $3$. 
 The set $D(M)$ is finite. 
\end{lemma}
\begin{proof} It follows because every element in $D(M)$ has distance at most $s$.
\end{proof}

\begin{lemma}\label{inny}
Let $M$ be a matrix satisfying Assumption $3$.
 Then, for every $m,n$, $w(n,m)$ is a linear combination of elements $u*w(n,m)$ with $u\in D(M)$.  
\end{lemma}
\begin{proof}
 It follows from Lemma \ref{inclu}, since every $t_{i}\in D(M)$.
\end{proof}

Fix $m$. Let $u\in D(M)$. Recall that $u*w(n,m)$ is a matrix with coefficients in $A$. By $u*w(n,m)_{k,l}$ we denote the element of $A$ which is at the $k,l$ entry of matrix $u*w(n,m)$.
By $[u*w(n,m)]_{k,l}$ we will
 mean the quintuple  $(u,n,m,k,l)$.

 Recall that by $I$ we denote the identity element in $H'$ (which can be also seen as the identity matrix when we embed $H$ into a matrix ring.)

\begin{delfin} (Definition of  ordering)

We first denote an ordering on elements of $E$ : $E_{1}<E_{2}< \ldots <E_{\beta }$.
 We then define $E_{\beta }<I -e<e<P_{1}$ and $P_{1}<P_{2}< \ldots  <P_{\beta '}$.
 We can now define a lexicographical  ordering on the good set $D(M)$.

In particular, if  the distance of $v$ is larger than the distance of $u$ then $v<u$;   for example $(I-e, e, I, \ldots )<(P_{1}, I, I, \ldots )$.

 Let $M$ be a matrix satisfying Assumption $3$. 
Fix $m$, and  define the following ordering on quintuples $[u*w(n,m)]_{k,l}$ with $u\in D(M)$, and $k,l\leq d$ where $M$ is a $d$ by $d$ matrix:
\begin{itemize}

\item[1.]  If the distance of $v$ is larger than the distance of $u$ then  $[v*w(n,m)]_{k,l}<[u*w(n',m)]_{k',l'},$ for every $n,n', k,k',l,l'\leq d$.

\item[2.]  If the distance of $v$ is the same as the  distance of $u$ and $n<n'$ then 
$[v*w(n, m)]_{k,l}<[u*w(n', m)]_{k',l'},$ for every $k,k',l,l'\leq d$.

\item[3.]  If the distance of $u$ is the same as the distance of $v$ and $v$ is smaller than $u$ then 
$[v*w(n, m)]_{k,l}<[u*w(n,m)]_{k',l'},$ for every $k,k',l,l'\leq d$.

\item[4.] If $(k,l)<(k',l')$ with respect of lexicographical ordering then   
$[u*w(n, m)]_{k,l}<[u*w(n,m)]_{k',l'},$ for every $u, n$.

\end{itemize}  
\end{delfin}
 Fix $m$. Notice that this is an ordering on the set of quintuples $[u*w(n,m)]_{k,l}$ with $u\in D(M)$ and $k,l\leq d$, where $M$ is a $d$ by $d$ matrix. 

\section{Sets $B_{m}(M)$ and $Z_{m}(M)$}

 We now define sets $B_{m}(M)$ and $Z_{m}(M)$. 

\begin{delfin} (Definition of  set $B_{m}(M)$) 
Let $M$ be a matrix satisfying Assumption $3$, and $m, n$ be natural numbers.  

We will say that quintuple $[u*w(n, m)]_{k,l}$ is in the set $B_{m}(M)$ if  $u*w(n,m)_{k,l}$  is a linear combination over $F$ of elements $v*w(n',m)_{k',l'}$
 such that   $[v*w(n',m)]_{k',l'}<[u*w(n,m)]_{k,l}$. 
\end{delfin}

\begin{delfin}\label{z} (Definition of  set $Z_{m}(M)$)
Let $M$ be a matrix satisfying Assumption $3$, and $m, n$ be  natural numbers.  

Recall that $R$ is the algebra generated by elements $a$ and $b$. 

We will say that quintuple $[u*w(n,m)]_{k,l}$ is in the set $Z_{m}(M)$  if there is element $r\in R$ such that:

\begin{itemize}  
\item[1.] This element $r$ is a linear combination of 
elements $v*w(n',m)_{k',l'}$
 such that   $[v*w(n',m)]_{k',l'}\leq [u*w(n,m)]_{k,l}$.

\item [2.] $r$ is not a linear combination of elements $v*w(n',m)_{k',l'}$
 with $[v*w(n',m)]_{k',l'}< [u*w(n,m)]_{k,l}$.
\end{itemize}

 Notice that for a given  quintuple $[u*w(n,m)]_{k,l}=(u,n,m,k,l)\in Z_{m}(M)$, there may be many elements $r$  
 satisfying Properties  [1] and [2] above. However, we fix one  such element $r$ and call it 
 $r(u, n, m,k,l)$. 

\end{delfin}

\begin{lemma}\label{zand} Fix $m$. Let $M$ be a matrix satisfying Assumption $3$.
  The sets $Z_{m}(M)$ and $B_{m}(M)$ are disjoint, that is  
$Z_{m}(M)\cap B_{m}(M)=0$.
\end{lemma}
\begin{proof} Suppose on the contrary, that  there is some element in $[u*w(n,m)]_{k,l}\in Z_{m}(M)\cap B_{m}(M)$, then there is  $r\in R$ which is a linear combination of some elements  $v*w(n',m)_{k',l'}$
 such that   $[v*w(n',m)]_{k',l'}\leq [u*w(n,m)]_{k,l}$. Because $[u*w(n,m)]_{k,l}\in B_{m}(M)$ then  
 $u*w(n,m)_{k,l}$  is a linear combination 
 of $v*w(n',m)_{k',l'}$
 such that   $[v*w(n',m)]_{k',l'}<[u*w(n,m)]_{k,l}$. Therefore 
$r$ is also a linear combination 
 of $[v*w(n',m)]_{k',l'}$
 such that   $[v*w(n',m)]_{k',l'}<[u*w(n,m)]_{k,l}$, a contradiction with the definition of  set $Z_{m}(M)$.
\end{proof}

\begin{lemma}\label{cardinality} Let notation be as above.
 Let $m$ be a natural number and let $R_{m}(M)$ be the linear space spanned by all elements from $R$ which are a linear combination of elements $u*w(n,m)_{k,l}$ for some $u\in D(M)$ and some $n,k,l$ with $k,l\leq d$ (where $M$ is a $d$ by $d$ matrix).  Then the dimension of the space $R_{m}(M)$ is the same as the cardinality of set $Z_{m}(M)$. 
\end{lemma}
\begin{proof} Let $r(u,n, m,k,l)\in R$ be as in Definition \ref{z}.
Let $Q_{m}(M)$ be the linear space spanned by elements $r(u,n, m,k,l)\in R$ for $(u,n,m,k,l)\in Z_{m}(M)$. We will show that $R_{m}(M)=Q_{m}(M)$. 
Observe first that if $s\in R_{m}(M)$ then  \[s=\sum_{(v,n',m,k',l')\leq (u, n, m, k, l)}\alpha _{(v,n',m,k',l')}\cdot v*w(n',m)_{k',l'},\]
 for some $(u, n, m, k, l)=[u*w(n,m)]_{k,l}$ and some $\alpha _{(v,n', m, k',l')}\in F.$ If we take a presentation of $s$ with $[u*w(n,m)]_{k,l}$ minimal possible, we in addition get  $\alpha _{(u,n,m,k,l)}\neq 0$ and $[u*w(n,m)]_{k,l}\in Z_{m}(M)$.

 Note that if $s=r(u,n,m,k,l)$, then $s\in Z_{m}(M)$. Suppose that $s\neq r(u,n,m,k,l)$; then by Definition \ref{z} there is $\alpha \in F$ such that 
\[s-\alpha \cdot r(u,n,m,k,l)=\sum_{(v,n',m,k',l')<(u, n, m, k, l)}\beta _{(v,n',m,k',l')}\cdot v*w(n',m)_{k',l'}\]
 for some $\beta _{(v,n',m,k',l')}\in F$. Therefore $s-r(u,n,m,k,l)\in R_{m}(M)$.

We will now show that $s\in Q_{m}(M)$ by induction with respect to the ordering of the quintuples $(u,n,m,k,l)\in Z_{m}(M)$.
 
 Let $[u_{0}*w(i,m)]_{t,t'}$ be the minimal quintuple in $Z_{m}(M)$ such that there is $s\in R_{m}(M)$ with $s\notin Q_{m}(M)$. Notice that $r(u_{0},i,m,t,t')\in R_{m}(M)$. Let $s\in R_{m}(M)$ have a presentation as above with $(u,n,m,k,l)=(u_{0},i,m,t,t')$.
  By the definition of $Z_{m}(M)$ we get that for some $\alpha $ the element 
 $s-\alpha \cdot r(u_{0},i,m,t,t')$ is a sum of elements associated to quintuples smaller than $(u_{0},i,m,t,t')$. Notice that 
$s-\alpha \cdot r(u_{0},i,m,t,t')\in R_{m}(M)$, and  by the minimality of 
 $s$ we get that $s-\alpha \cdot r(u_{0},i,m,t,t')\in Q_{m}(M).$
 Since $ r(u_{0},i,m,t,t')\in Q_{m}(M)$, it follows that $s\in Q_{m}(M)$.
%
\end{proof}

\section{Main supporting lemma}

 Let $M$ be a matrix which satisfies Assumption $3$, and let notation be as in the previous section.

\begin{lemma}\label{important5} Let $M$ be a matrix satisfying Assumption $3$. 
Let $m$ be a natural number and let $u=(f_{1}, f_{2}, \ldots )\in D(M).$ For every $k\leq m$, 
\[u*w(n,m)=\sum_{j=0}^{n}((f_{1}, f_{2}, \ldots , f_{k},I,I, \ldots )*w(j,k))\cdot ((f_{k+1}, f_{k+2}, \ldots )*w(n-j, m-k)).\]
\end{lemma}
\begin{proof} This follows from Lemma \ref{known3} and from the definition of operation $*$.
\end{proof}
\begin{lemma}\label{useful} Let notation be as in Lemma \ref{important5}. Suppose that $u$ has distance $k+1$. Then
\[u*w(n,m)=\sum_{j=0}^{n}((f_{1}, f_{2}, \ldots , f_{k})*w(j,k))\cdot f_{k+1}\cdot w(n-j, m-k).\]
Moreover, for every $t\leq m-k$,
\[w(n-j,m-k)=w(0,t)w(n-j,m-k-t)+\sum_{i=1}^{n-j}w(i, t)w(n-j-i,m-k-t).
]\]
\end{lemma}
\begin{proof} Since $u$ has distance  $k+1$ then  $u=(f_{1}, f_{2}, \ldots ,f_{k}, f_{k+1}, I, I, \ldots , I)$, where $f_{k+1}\in \{e, P_{1}, \ldots , P_{\beta '}\}$  and by assumptions on $P_{i}$'s we have $f_{k+1}e=f_{k+1}$. The first equation follows from Lemma \ref{important5}. 
   The second equation follows when we apply Lemma \ref{known3} to $w(n-j, m-k)$.
\end{proof}

 Let $M$ be a matrix satisfying Assumption $3$, and let notation be as in Assumption $3$. In particular, $H'$ is an algebra generated by matrices from $H$ and an identity matrix,  $W$ is the Wedderburn radical of $H$ and $W^{s}=0$. Recall that $I$ denotes the identity element in $H'$.

\begin{lemma}\label{previous} Let notation be as above. 
 Let $f_1, f_2, \ldots , f_{k}\in W\cup \{I-e\}$ and let $f_{k+1}\in H'$.
 Let $u=(f_1, f_{2}, \ldots , f_{k+1}, I, I, \ldots )$, and let $m,n>0$, $m>k+1,$ $m>s$. Then 
$u*w(n,m)$  is a linear combination of matrices of the form $u'*w(n,m)$ and $u''*w(n,m)$, where $u'\in D(M)$ and $u'$ 
 has distance $k+1$, and  where $u''$ is of the form $(g_{1}, \ldots , g_{k+1}, I, I, \ldots )$ where $g_{1}, g_{2}, \ldots , g_{k+1}\in E$. 
\end{lemma}
\begin{proof} Observe that $f_{k+1}$ is a linear combination of elements $\alpha _{i}$ and $\beta _{i}$, where $\alpha _{1} \in E=\{E_{1}, \ldots , E_{\beta }, I-e\}$ and $\beta _{i}\in \{e, P_{1}, \ldots , P_{\beta '}\}$
(where notation is as in the definition of the ordering of $D(M)$). Therefore 
$u*w(n,m)$ is a linear combination 
of elements  $u(i)*w(n,m)$ and elements $u(i)*w(n,m)$, where $u(i)=(f_{1}, \ldots , f_{k}, \alpha _{i}, I, I, \ldots )$ and elements $u'(i)=(f_{1}, \ldots , f_{k}, \beta _{i}, I, I, \ldots )$, for some $i$. Observe that $u'(i)$ are in the set $D(M)$ and  have distance $k+1$. On the other hand,  elements $u(i)$ are of the form 
$(g_{1}, \ldots , g_{k+1}, I, I, \ldots )$, where $g_{1}, g_{2}, \ldots , g_{k+1}\in W\cup \{I-e\}$ and $g_{k+2}=I$, as required.
\end{proof}


\section{ Introducing sets $U_{k,t}$ and $V_{k,t}$}

Let $M=\sum _{i=1}^{\xi }A_{i}a_{i}$ be a matrix satisfying Assumption $3$. Let notation be as in Assumption $3$ and in the previous sections. Let $a_{\xi +1}, a_{\xi +2}, \ldots $ be such that $a_{1}, a_{2}, \ldots , a_{\xi}, a_{\xi +1}, \ldots $ is a basis of 
 $A(\alpha )$ (such elements exist by Zorn's lemma). Recall that $a_{1}, \ldots , a_{\xi }\in \langle x\rangle $; hence we can assume that every $a_{i}$ is either in $\langle x\rangle $ or in $R$. Recall that $\langle x\rangle $ is the ideal of $A'$ generated by $x$.

Denote $Q_{t}=\sum _{(i_1, i_2, \ldots , i_{t})\neq (1,1, \ldots ,1)}a_{i_1}a_{i_{2}}\ldots a_{i_t}A.$ 
Let $k, t$ be  natural numbers, and define:
\[V=A(k)a_{1}^{t}A, U=A(k)Q_{t}.\]
Recall that $M$ is a $d$ by $d$ matrix. 
 By $T(U)$ we will denote the set of all $d$ by $d$ matrices with all entries in $U$,  and by $T(V)$ we will denote the set of all $d$ by $d$ matrices whose entries are in $V$.
Recall that $s$ is such that $W^{s}=0$.

\begin{remark}\label{new}
Let $r\in R\cap A(m)$ for some $m>t+k$. Then $r\in U$. 
\end{remark}
\begin{proof} Denote $Q= A(k)(R\cap A(t))A(m-t-k)$, $Q'=A(k)(\langle x\rangle\cap A(t))A(m-t-k).$ Notice that $Q\cap Q'=0$, since $A$ is the free algebra generated by $ax^{i}$, $bx^{i}$ for $i\geq 0$. Let
$U_{1}=U\cap Q$ and $U_{2}=U\cap Q'.$ Observe that  $U\cap A(m)=U_{1}\oplus U_{2}$, 
 because every element among $a_{1}, a_{2}, \ldots $ is either in $R$ or in $\langle x\rangle $.

Notice that if $a\in A(m)$ then $a\in U_{1}+U_{2}+V$.
 Recall that  $U_{2}, V\subseteq \langle x\rangle$, since $a_{1}\in \langle x\rangle $. Let $r=r_{1}+r_{2}+r_{3}$, where $r_{1}\in U_{1}, r_{2}\in U_{2}, r_{3}\in V.$
 Observe that  $r-r_{1}\in Q$ and $r_{2}+r_{3}\in  Q'$; it follows that $r=r'\in U_{1}\subseteq U$.
\end{proof}

Let ${\mathcal G}:A(m)\cap V\rightarrow A(m-t)$  be the linear mapping defined for monomials and then extended by linearity to all elements from $A(m)\cap V$ as follows:
 
 ${\mathcal G}(w{a_1}^{t}w')=ww'$, where $w$ is a monomial from $A(k)$ and $w'$ is a monomial from $A(m-k-t)$. 
  We can then  extend mapping ${\mathcal G}$ to matrices: if $M$ is a matrix with entries $a_{i,j}$ then ${\mathcal G}(M)$ is the matrix with entries ${\mathcal G}(a_{i,j})$.  

\begin{lemma}\label{G}
 Let $n,m,k,t$ be natural numbers with $m>k+t, m>t+s$, $t\geq 1$. Let ${\mathcal G}$ be defined as before this theorem. Let $u=(f_{1}, f_{2}, \ldots )\in D(M)$ have  distance $k+1$, so 
 $f_1, \ldots, f_{k}\in E$ and $f_{k+1}=he$ for some $h\notin W$. 
 Then 
$u*w(n,m)=\bar {v}+\bar u$  for some $\bar {v}\in T(V)$, $\bar {u}\in T(U)$.  
Moreover,  ${\mathcal G}({\bar v})=u*w(n,m-t)+s$, where $s$ is a linear combination of elements of the form $u'*w(n-i,m-t)$ for $i>0$ and where $u'\in D(M)$ has either the same distance as $u$ or larger distance than $u$.
\end{lemma}
\begin{proof} Observe that  $w(n,m)=\sum_{i=0}^{n}F'_{i}$ by Lemma \ref{useful}, where  \[F'_{i}=\sum_{j=0}^{n-i}w(j,k)w(i,t)w(n-i-j,m-k-t).\]
 By the definition of operation $*$ we get 
 $u*w(n,m)=\sum _{i=0}F_{i}$ where 
\[F_{i}=\sum_{j=0}^{n-i}[u'*w(j,k)]\cdot f_{k+1}\cdot w(i,t)\cdot w(n-i-j,m-k-t),\] 
 where $u'=(f_1,f_2, \ldots ,f_k, I, I\ldots )$. Recall also that $f_{k+1}=he$ for some $h\notin W$.

Observe that  \[w(j,t)=q_{j}{a_1}^{t}+u'_{j},\] 
 for some $u'_{j}\in T(U)$ and some matrix $q_{j}$ with entries in $F$. Recall that $w(0,t)={M_{0}}^{t}=M^{t}=(\sum_{i=1}^{\xi }A_{i}a_i)^{t}$ and that $A_{1}=e$. 
Therefore $q_{0}={A_{1}}^{t}=e$, so $w(0,t)=e{a_{1}}^{t} + u'_{0}$. 
It follows that 
$F_{0}=v_{0}+u_{0}$, where $u_{0}\in T(U)$ and \[v_{0}=\sum_{j=0}^{n}[u'*w(j,k)]\cdot f_{k+1}e{a_1}^{t}\cdot w(n-j,m-k-t)\in T(V).\] 
 Recall that $f_{k+1}e=f_{k+1}$, by assumption on $P_{i}$'s.   By Lemma \ref{known3}, \[{\mathcal G}(v_{0})= \sum_{j=0}^{n}[u'*w(j,k)]\cdot f_{k+1}e\cdot w(n-j,m-k-t)=u*w(n,m-t).\] 

Observe now that 
$F_{i}=v_{i}+u_{i}$, where $u_{i}\in T(U)$ and   \[v_{i}=\sum_{j=0}^{n-i}[u'*w(k,j)]\cdot f_{k+1}\cdot q_{i}{a_{1}}^{t}\cdot w( m-k-t, n-i-j) \in T(V).\]  
 By Lemma \ref{known}, 
 \[ {\mathcal G}(v_{i})=u(i)*w(,n-i, m-t)\] where $u(i)=(f_{1}, f_{2}, \ldots , f_{k}, f_{k+1}q_{i}, I, I, \ldots  )$.
 By applying Lemma \ref{previous}  several times we get that  $u(i)*w(n-i, m-t)$ is a linear combination of elements of the form $u'*w(n-i, m-t)$ for $i>0$ and where $u'$ is in $ D(M)$ and has   distance at least $k+1$, or $u'=(g_{1}, \ldots g_{l}, I, I, \ldots )$ for some $l>s$ and all $g_{1}, \ldots , g_{l}\in E$. In the latter case $u'*w(n-i, m-t)=0$ by Lemma \ref{zerowy}, hence the latter case can be omitted. 
 Observe now that $\bar {v}=\sum_{i=0}^{n}v_{i}$, so ${\mathcal G}(\bar {v})={\mathcal G}(v_{0})+\sum _{i}{\mathcal G}(v_{i})=u*w(n, m-t)+\sum _{i=1}^{n}u(i)*w(n-i,m-t)$ and the result follows (since $u$ has distance $k+1$). 
\end{proof}
\begin{lemma}\label{Gdluzszy}
 Let $m,k,t, n$ be natural numbers with $m>k+t$, $m>t+s$ and $t\geq s$, $t\geq 1$ (where $W^{s}=0$). Let $u=(f_{1}, f_{2}, \ldots )\in D(M)$ have  distance larger than  $k+1$, so 
 $f_1, \ldots, f_{k+1}\in E$. 
 Then 
$u*w(n,m)=\bar {v}+\bar u$  for some $\bar {v}\in T(V)$, $\bar {u}\in T(U)$.  

Let  ${\mathcal G}:A(m)\cap V\rightarrow A(m-t)$  be as defined and extended to matrices as in Lemma \ref{G}. 
Then ${\mathcal G}({\bar v})$ is a linear combination of elements of the form $u'*w(n', m-t)$ for $n'\geq 0$ and where $u'\in D(M)$ has distance larger than $k+1$. 
\end{lemma}
\begin{proof} Observe that  $w(n,m)=\sum_{i=0}^{n}F'_{i}$ by Lemma \ref{useful}, where  \[F'_{i}=\sum_{j=0}^{n-i}w(j,k)w(i,t)w(n-i-j,m-k-t).\]
 By the definition of operation $*$ we get 
 $u*w(n,m)=\sum _{i=0}F_{i}$ where 
\[F_{i}=\sum_{j=0}^{n-i}[u'*w(j,k)]\cdot u''*\cdot w(i,t)\cdot w(n-i-j, m-k-t),\] 
 where $u'=(f_1,f_2, \ldots ,f_k, I, I\ldots )$ and
 $u''=(f_{k+1}, f_{k+2}, \ldots , f_{s}, I, I, \ldots )$  (as elements of $D(M)$ have distance at most $s$).

Observe that  \[u''*w(j,t)=q_{j}{a_1}^{t}+u'_{j}\] 
 for some $u'_{j}\in T(U)$ and some matrix $q_{j}$ with entries in $F$.
 Notice also that $q_{j}=f_{k+1}h$ for some matrix $h$ since $f_{k+1}$ is the first entry of $u''.$ It follows that $q_{j}\in W$ since $f_{k+1}\in W\cup \{I-e\}.$

  Observe now that 
 $F_{i}=v_{i}+u_{i}$, where  \[v_{i}=\sum_{j=0}^{n-i}[u'*w(j,k)]\cdot  q_{i}{a_{1}}^{t}\cdot w( n-i-j, m-k-t),\in T(V)\] and $u_{i}\in T(U)$. 
 By Lemma \ref{known3}, 
 \[{\mathcal  G}(v_{i})=u(i)*w(n-i, m-t)\] where $u(i)=(f_{1}, f_{2}, \ldots , f_{k}, q_{i}, I, I, \ldots  )$.
 Recall that $q_{i}\in W$.
 Because $f_1,\ldots , f_{k+1}\in W$, then  by Lemma \ref{previous} applied several times we get that  $u(i)*w(n-i,m-t)$ is a linear combination of elements of the form $u'*w(n-i,m-t)$ for $i>0$ and where either $u'$ is in $ D(M)$  and has  the distance larger than  $k+1$, or $u'=(g_{1}, \ldots g_{l}, I, I, \ldots )$ for some $l>s$ and all $g_{1}, \ldots , g_{l}\in E$. In the latter case $u'*w(m,n)=0$ by Lemma \ref{zerowy}, hence this case can be omitted. 
 Observe now that $\bar {v}=\sum_{i=0}^{n}v_{i}$, so ${\mathcal G}(\bar {v})=\sum _{i=0}^{n}{\mathcal G}(v_{i})=\sum _{i=0}^{n}u(i)*w(n-i,m-t)$; the result follows (since each $u(i)$ has distance larger than $k+1$). 
\end{proof}

Recall that $s$ is such that $W^{s}=0$.
\begin{theorem} \label{naj} Let $M$ be a matrix satisfying Assumption $3$. 
Let $t>s$, and $k,l>0, m>k+t, m>t+s,  n\geq 0$ be natural numbers. Let $u\in D(M)$. If element $[u*w(n,m)]_{j,l}\in Z_{m}(D)$ then $[u*w(n,m-t)]_{j,l}\in B_{m-t}(M)$.
\end{theorem}
\begin{proof}
 Since $u\in D(M)$ then $u$ has distance $k+1$ for some $k\geq 0$. By the definition of set $Z_{m}(M)$ there is  $r\in R$ such that \[r=\sum_{(v,n',m, j',l')\leq (u,n,m,j,l)}\alpha _{(v,n',m,j',l')}v*w(n',m)_{j',l'}\] where $\alpha _{(v,n',m,j',l')}\in F$ and 
$\alpha _{(u,n,m,j,l)}\neq 0$. Observe that for every $v\in D(M)$ and every $n'$ we can write  
\[v*w(n',m)=q(v,n',m)+z(v,n',m)\] where $q(v,n',m)\in T(V)$ and 
$z(v,n',m)\in T(U)$, where  $T(U), T(V)$ are as in Lemmas \ref{G} and \ref{Gdluzszy}. By $q(v,n',m)_{j,l}$  we will denote the $k,l$-entry of matrix  $q(v,n',m)$ (similarly for $z(v,n',m)$). It follows that 
\[r=\sum_{(v,n',m,j',l')\leq (u,n,m,j,l)}\alpha _{(v,n',m,j',l')}(q(v,n',m)_{j',l'}+z(v,n',m)_{j',l'}).\]
 By Remark \ref{new} $r\in U$, hence $\sum_{(v,n',m,j',l')\leq (u,n,m,j,l)}\alpha _{(v,n',m,j',l')}q(v,n',m)_{j',l'}\in U$.
 Since $U\cap V=0$ it follows that 
\[\sum_{(v,n',m,j',l')\leq (u,n,m,j,l)}\alpha _{(v,n',m,j',l')}q(v,n',m)_{j',l'}=0.\]
 We can apply mapping ${\mathcal G}$ to this equation. We then get 
\[{\mathcal G}(q(v,n,m)_{j,l})= \sum_{(v,n',m,j',l')<(u,n,m,j,l)}\beta _{(v,n',m,j',l')}{\mathcal G}(q(v,n',m)_{j',l'}),\]
 for some $\beta _{(v,n',m,j',l')}\in F.$
Let $W$ be the linear space spanned by all elements $v*w(n',m-t)_{j',l'}$ with $[v*w(n',m-t)]_{j',l'}<[u*w(n,m-t)]_{j,l}.$ 
By Lemma \ref{G}, ${\mathcal G}(q(u,n,m)_{j,l})-u*w(n,m-t)_{j,l}\in W$. By Lemmas \ref{Gdluzszy} and \ref{G},
 ${\mathcal G}(q(v,n',m)_{j',l'})\in W$, provided that
  $[v*w(n',m)]_{k',l'}<[u*w(n,m)]_{j',l'}$ (if  $v$ has the same distance as $u$ then we use Lemma \ref{G}; if $v$ has distance larger than $u$ we use Lemma \ref{Gdluzszy}).
 Therefore, $u*w(n,m-t)_{j,l}\in W$. This means that $[u*w(n,m-t)]_{j,l}\in B_{m-t}(M)$.
\end{proof}


\section{Main result}
Let $N$ be a matrix. Recall that by $S(N)$ we denote  the linear space spanned by all entries of matrix $N$; similarly if $N_{1}, N_{2}, \ldots , N_{k}$ are matrices with entries in $A$ then by
$S(N_{1}, N_{2}, \ldots , N_{k})$ we will denote the linear space spanned by all entries of matrices $N_{1}, N_{2}, \ldots , N_{k}$.

\begin{theorem}\label{with}
  Let $M$ be a matrix satisfying Assumption $3$.  For arbitrary $c$, there is $m>c$ such that 
   $R\cap S(w(0,m), w(1,m), w(2,m), w(3,m), \ldots )$ is a linear space over $F$ of dimension less than $\sqrt n$.
\end{theorem}
\begin{proof} Recall that, by Lemma \ref{inny},  $w(0,m), w(1,m), \ldots \in \sum_{u\in D(M), i=0,1, \ldots }Fu*w(i,m)$ for $i=1,2, \ldots k$ and $u\in D(M)$
 with distance at most $s$, where $W^{s}=0$. 
 It is sufficient to show that for infinitely many $m$ the dimension of the set $R_{m}(M)=R\cap \sum _{u\in D_{m}(M), i=0,1, \ldots }S(u*w(i,m))$ is less than $\sqrt m$. By Lemma \ref{cardinality} it is equivalent to show that the cardinality of set $Z_{m}(M)$ is smaller than $\sqrt {m}$ for infinitely many $m$.

We will provide a proof by contradiction. 
 Suppose, on the contrary,   that there is $c$ such that  for every $m>c$, set $Z_{m}$ has more than $\sqrt m$ elements.

By Theorem \ref{naj}, if $[u'*w(i,m)]_{k,l}\in Z_{m}(M)$ and $[u*w(i',m')]_{k',l'}\in Z_{m'}(M)$ and $m>m'+s,$ $m'>2s$ then $(u, i, k,l)\neq (u', i', k', l')$. It follows because by Theorem \ref{naj} $[u*w(i,m')]_{k,l}\in B_{m'}(M)$ and 
$B_{m'}(M)\cap Z_{m'}(M)=0$ by Theorem \ref{zand}, so $(u, i, k,l)\neq (u', i', k', l')$.

Let $m$ be a natural number. Recall that $M$ is a $d$ by $d$ matrix with entries in $A(\alpha )\cap \langle x\rangle $. 
 Recall that $\gamma _{y}(M)=\sum_{i=0}^{t}M_{i}y^{i},$ so $w(i,m)=0$  if $i>tm$.

Let $C_{m}(M)$ be the set of all tuplets $(u, i, m, k,l)$, where $0\leq i\leq m(s+2)t $, $k,l\leq d$, $u\in D(M)$. 
  Recall that set $D(M)$ is finite. Therefore there is a constant $z$ such that for every $m$ the cardinality of the set $C_{m}(M)$ is smaller than $zm$.

We can now choose $m>z^{2}$ and $m>c+3s$. 

For $q=1,2, \ldots, m$ let $F_{q}$ be the set of elements $(u,i, m, k,l)$ such that $(u, i, m+q(s+1), k,l)\in Z_{m+q(s+1)}(M)$ where $u\in D(M)$.  

 Let $(u, i, m+q(s+1), k,l)\in Z_{m+q(s+1)}(M)$. Notice that $i\leq  t(m+q(s+1)) \leq m\cdot t\cdot (s+2) $,  as otherwise $w(i, m+q(s+1))=0$ (since $w(i,j)$ is zero if $i>t\cdot j$ by the construction of $w(i,j)$). Therefore $F_q\subseteq C_{m}(M)$, for  $q=1,2, \ldots ,m$.
 
Notice that the cardinality of $F_{q}$ is the same as the cardinality of  $Z_{m+q(s+1)}(M)$, and hence larger than $\sqrt m$.
By Theorem \ref{naj}, $F_{i}\cap F_{j}=\emptyset $ for any $1\leq i,j\leq m$.
 Therefore the cardinality of $\bigcup _{i=1}^{s}F_{i}$ is larger than $m{\sqrt m}$.

 Recall that $F_{i}\subseteq C_{m}(M)$ for $i=1,2, \ldots ,m$.  
 Therefore, the cardinality of 
 $C_{m}(M)$ has to be at least $m\sqrt m$. This gives us a contradiction, since we showed that the cardinality of $C_{m}(M)$ is  smaller than $zm$, yet we assumed that $m>z^{2}$. 
\end{proof}

We will now prove that Assumption $1$ holds for algebras over a field $F$, where $F$ is the algebraic closure of a finite field. 

\begin{theorem}\label{slm} Let $F$ be a field. Let $M$ be a matrix satisfying Assumption $3$.  For arbitrary $c$, there is $m>c$ such that 
   $R\cap S(L(M^{m}))$ is a linear space over $F$ of dimension less than $\sqrt n$.
\end{theorem}
\begin{proof} We can now apply Theorem \ref{with}
 to the matrix $M$ to get that  \[R\cap S(w(0,m), w(1,m), w(2,m), \ldots) \] is a linear space over $F$ of dimension less than $\sqrt n$. By Lemma \ref{wl},  we have \[S(L(M^{m}))=S(w(0,m), w(1,m), w(2,m), \ldots).\]
\end{proof}  

\section{Matrices}
 In this section $F$ denotes the algebraic closure of a finite field.
 The aim of the next two sections is to show that, if $N$ is an arbitrary matrix with entries in $A(j)\cap \langle x\rangle $ for some $j$, then some power of $N$ satisfies  Assumption $3$.
Here the notation $A$ and $R$ is not related to the similar notation appearing in previous chapters; instead, $R$ denotes a general ring. 
\begin{delfin}\label{17}
  Let $R$  be a finite dimensional  $F$-algebra generated by elements  $r_1, r_2, \ldots , r_{n}\in R$. Let $M$ be the multiplicative monoid generated by elements $r_{1}, \ldots , r_{n}$. Let $Z^{+}$ be the set of all positive integers. 
 Let $\alpha :M\rightarrow Z^{+}$ be the function such that 
\begin{itemize}
\item $\alpha (r_{i})=1$ for $i=1, \ldots , n$.
\item If $u,v\in M$ then $\alpha (u\cdot v)=\alpha (u)+\alpha (v).$
\end{itemize} 
  The number $\alpha (r)$ will be caled the weight of element $r\in M$.
 Notice that one element may have many weights. 

 We say that $\alpha $ is the weight function on $R$ related to elements $r_{1}, \ldots , r_{n}$. 
 \end{delfin}
 \begin{delfin} Let notation be as in Definition \ref{17}.
 We will say that an element $r\in R$ is pseudo-homogeneous if it can be expressed as a linear combination of elements with the same weight, say $\beta $. The weight of $r$ is $\beta $. The weight of $r$ will be denoted $\deg_{w} (r)$.

 By the linear space of pseudo-homogeneous elements of weight $n$ we will mean the linear space over $F$ spanned by all pseudo-homogeneous elements of weight $n$; this linear space will be denoted $R(n)$.
\end{delfin}
 By $R^{1}$, we will denote the algebra which is the usual extension of $R$ by an identity element, and by $R(0)$ we will denote the space $F\cdot 1$ in $R^{1}$.

The following Lemma closely resembles Lemma $1$ from \cite{ps}. However, our ring need not be graded, so we provide a detailed proof using similar methods as in \cite{ps}.

\begin{lemma}\label{mie}  Let  $R$ be a simple $F$-algebra with an identity element and let notation be as in Definition \ref{17}.
 Assume that  \[1=\sum_{i=k}^{k'}b_{i},\] where $b_{i}$ is a pseudo-homogeneous element of weight $i$ for each $i$. Let $h\in R$ be a pseudo-homogeneous element in $R$.  Denote $H_{i}= \sum_{j=0}^{i-\deg _{w}h}R(j)hR(i-j-\deg_{w} h)$  
for $n\geq \deg _{w} h$.
 Then there exist  $c_{i}\in  H(i)$  such that 
\[1=\sum_{i=t}^{t'}c_{i}\] for some $t> k'$ and $t'-t< k'$.
\end{lemma}
\begin{proof}  Let $I$  be the ideal generated by $h$ in $R$ then $I=\sum_{i=\deg _{w} h}^{\infty } H_{i}$. Notice that $R=I$ since $R$ is a simple algebra, therefore there are $c_{i}\in H(i)$ such that  $1=\sum_{i=t}^{t'}c_{i}$;  $t'-t$ will be called the length of expression $1=\sum_{i=t}^{t'}c_{i}.$  
 We can assume that $t>k$ and that $t'-t$ is minimal possible. If $t'-t< k'$ then the result follws; suppose that 
$t'-t\geq k'$. Recall that  $\sum_{i=k}^{k'}b_{i}=1$ therefore 
 \[1=c_{t}(\sum_{i=k}^{k'}b_{i})+\sum _{i=t+1}^{t'}c_{i}.\] Notice that the weight of element $c_{t}b_{k'}$ is $t+k'\leq t'$.
  Therefore, $1=\sum_{i=t+1}^{t'}e_{i}$ where $e_{i}=c_{t}b_{i-t}+c_{i}$ for $i\geq t+1$. We have obtained a contradiction since 
 the expression $1=\sum_{i=t+1}^{t'}e_{i}$ has smaller length than the expression $1=\sum_{i=t}^{t'}c_{i}$.
\end{proof}
 The following lemma resembles Proposition $1$ from \cite{ps}. However, our ring is ungraded, so we need to repeat the argument. 
\begin{lemma} Let $R$ be a simple $F$-algebra with an identity element and let notation be as in Definition \ref{17}.  
  Let \[1=\sum _{i=k}^{k'}b_{i}\] with each $b_{i}$ pseudo-homogeneous of weight $i$ for some $k,k'$, with $k'-k$ the minimal possible.
 Then all $b_{i}$ are in the center of $R$.
\end{lemma}
\begin{proof} The proof is similar to the proof of Proposition $1$ in \cite{ps}.
 We will show first that all $b_{i}$ belong to the center of $R$. Suppose  the contrary, and let $z$ be minimal such that  $c'=rb_{z}-b_{z}r\neq 0$ for some pseudo-homogeneous $r$ (we can assume that $r$ is pseudo-homogeneous, since every element in $R$ is a linear combination of pseudo-homogeneous elements).  Since $R$ is simple, then $R$ equals the ideal generated by $c'$. Notice that $c'$ is a pseudo-homogeneous element. Let $c$ be a pseudo-homogeneous 
  element which is a product of the generators $r_{1}, \ldots , r_{n}$ of $R$ and the  element $c'$, with the element $c'$ appearing at least $2k'$ times. Such a non-zero element $c$ exists, since a simple ring is prime. 
 
 By the previous lemma, \[1=\sum_{j=t}^{t'}c_{i}\]  where $c_{i}$ are pseudo-homogeneous and $c_{i}\in  \sum_{j=0}^{i-\deg_{w} c}R(j)cR(i-j-\deg_{w} c)\subseteq R_{i}$. 
 Moreover, 
 $t'-t< k'$. Recall that $c'=-\sum_{i=z+1}^{k'}d_{i}$, with $d_{i}=rb_{i}-b_{i}r$.
 
Notice that $c_{t}$ is pseudo-homogeneous of weight $t$. Observe that $c'$ is pseudo-homogeneous of weight $z+\deg_{w} r$ and each $d_{i}$ is pseudo-homogeneous of weight $i+\deg_{w} r$. Recall that each $c_{i}$ is a linear combination of  products of elements $a_{i}$ and element $c'$; we can substitute at some place in each of these products  
 $c'=-\sum_{i=z+1}^{k'}d_{i}$.  
 Therefore, \[c_{t}=\sum_{i=t+1}^{ k'-z+t}f_{j}\]  where each  $f_{j}$ is a linear combination of  products of some generators $r_{1}, \ldots , r_{n}$ of $R$ and  element $c'$, with element $c'$ appearing at least $2k'-1$ times and each product is  pseudo-homogeneous of weight $j$ (so each $f_{j}$ is pseudo-homogeneous of weight $j$). Now if $k'-z+t\leq t'$ we can substitute in this way for element $c_{t}$ in $1=\sum_{j=t}^{t'}c_{i}$ and obtain 
 \[1=\sum_{j=t+1}^{t'}c'_{i}\] where each $c_{i}'$ is pseudo-homogeneous of weight $i$ and 
 is a linear combination of  products of the generators $r_{1}, \ldots , r_{n}$ of $R$ and the element $c'$, with the element $c'$ appearing at least $2l-1$ times.

 Continuing in this way we can  substiture $c'=-\sum_{i=z+1}^{k'}d_{i}$  several times to obtain  (because $t'-t\leq l$)
 that  \[1=\sum_{j=t'-(k'-z)+1}^{t'}c''_{i}\]
  where $c''_{i}$ are pseudo-homogeneous of weight $i$. Observe that $t'-(t'-(k'-z)+1)=k'-z-1< k'-k$. This is a shorter presentation than \[1=\sum _{i=k}^{k'}b_{i}.\] Hence we obtain a contradiction. 
\end{proof}
  
\begin{lemma}
 Let $R$ be a finite dimensional simple $F$-algebra and let notation be as in Definition \ref{17}. Then $1$ is a pseudo-homogeneous element of $R$.
\end{lemma}
\begin{proof}  Since $a_{1}, \ldots ,a_{n}$  generate $R$, then there are pseudo-homogeneous elements $b_{i}$ such that 
\[1=\sum _{i=k}^{k'}b_{i}.\] We can assume that  $k'-k$ is the minimal possible. 
 By the previous lemma, each $b_{i}$ is central. By the Wedderburn-Artin theorem, $R$ is isomorphic to a matrix ring with coefficients from $F$. Hence, every central element is of the form $\alpha \cdot  I$, where $I$ is the identity matrix and $\alpha $ is from $F$. 
 Then $p_{i}=\alpha _{i} \cdot I$, and since $1=\sum_{i=k}^{k'}b_{i}$, then some $\alpha _{i}\neq 0$. Then ${1\over \alpha _{i}}b_{i}=I$ is pseudo-homogeneous.
\end{proof}

\begin{remark}\label{remark}  Since $F$ is the algebraic closure of a finite field, then for every matrix $m$ there is a natural number $\gamma (M)>0$ such that 
 $(M^{\gamma (M)} )^{2}=M^{\gamma (M)}$ is a diagonalizable matrix. Moreover,
  if all eigenvalues of $M$ are nonzero, then there is a natural number  $\beta (M)>0 $ such that 
 $M^{\beta (M)}=I$, the identity matrix.

To prove this, we need to restrict ourselves to diagonal matrices, where this result holds, and to the Jordan blocks.
Let $\alpha I+N$ be a Jordan block with $\alpha $ on diagonal; then $N$ is a stricly uppertriangular  matrix, and hence nilpotent.
Let $p$ be a characteristic of the field $F$. Then $(\alpha I+N)^{p^{n}}=\alpha ^{p^{n}}I+N^{p^{n}}$; therefore $(\alpha I+N)^{p^{n}}=\alpha ^{n}I=I$ for sufficiently large $n$, as required. For some related results see \cite{alexiejczyk}.
\end{remark} 

 If $R_{1}, R_{2}, \ldots R_{t}$ are algebras, then elements of algebra $R=\oplus_{i=1}^{t}R_{i}$ will be written as 
$(q_{1}, q_{2}, \ldots , q_{t})$, with $q_{i}\in R_{i}$.
\begin{theorem}\label{ttt}   Let $R$ be a finite dimensional $F$-algebra and let notation be as in Definition \ref{17}.
  Suppose that  $R=R_{1}\oplus R_{2}\oplus \cdots \oplus R_{t}$  for some simple finite dimensional $F$-algebras  $R_{1}, R_{2}, \ldots ,R_{t}$. Then $1$ is a pseudo-homogeneous element of $R$. 
\end{theorem}
\begin{proof} We proceed by induction  on $t$. If $t=1$ then the result follows from the previous Lemma. 
 Assume that $t>1$ and that the result holds for numbers $1, 2, \ldots , t-1$. Recall that $r_{1}, \ldots , r_{n}$ are generators of $R$ (see Definition \ref{17}).
  Each element $r_{i}$ can be written as $r_{i}=(r'_{i}, e_{i})$ with $r_{i}'\in R'$ where $R'= R_{1}\oplus R_{2}\oplus \cdots \oplus R_{t-1}$ and 
$e_{i}\in R_{t}$.  We can apply the inductive assumption to the algebra $R'=R_{1}\oplus R_{2}\oplus \cdots \oplus R_{t-1}$ with 
 generators $a'_{i}$ for $i=1, \ldots , n$. Then the identity element of $R'$ is pseudo-homogeneous in $R'$.
It follows that  element $(1, e)$ is a pseudo-homogeneous element of $R$ for some $e\in R_{t}$. 

Similarly, by the previous Lemma applied to the ring $R_{n}$ we obtain that $(f, 1)$ is a pseudo-homogeneous element of $R$, for some $f\in R'$.

Observe that, since a power of a pseudo-homogeneous element is a pseudo-homogeneous element, then
$(1, e)^{\gamma }$ and  $(f, 1)^{\beta }$ are pseudo-homogeneous elements of 
   the same weight for some $\gamma ,  \beta >0$.  Let $\alpha _{1}, \ldots ,\alpha _{s}$ be the eigenvalues of matrix $f^{\beta }$; then for any scalar $c$ the matrix 
$cI+f^{\beta }$ has  the eigenvalues  $c+\alpha _{1},\ldots ,c+ \alpha _{s}.$ 
The field $F$ is infinite, therefore there are $c,c'\in F'$ such that all the eigenvalues of matrix $M= 
c(1, e)^{\gamma }+c'(f, 1)^{\beta }$ are nonzero. By  Remark \ref{remark}, some power of the matrix $M$ is the identity matrix. As $M$ is a pseudo-homogeneous element of $R$ it proves the result.
\end{proof}
We will now prove the following result for rings which are not necessarily semisimple.
\begin{theorem}  Let $R$ be a finite dimensional $F$-algebra which is not nilpotent and let notation be as in Definition \ref{17}.
 Then there is a pseudo-homogeneous element $e\in R$ such that $e^{2}-e\in W$, and  for every $r\in R$ we have $r-er\in W$ and $r-re\in W$,
where $W$ is the Wedderburn radical of $R$ (the largest nilpotent ideal in $R$).
\end{theorem}
\begin{proof} Consider the algebra $R'=R/W$. Then $r'_{1}=r_{1}+W, r'_{2}=r_{2}+W, \ldots ,r'_{n}= r_{n}+W$ are generators of $R'$. 
 We can consider Definition \ref{17} for the algebra $R'$ and its generators $r'_{1}, \ldots , r'_{n}$ (in place of $R$ and $r_{1}, \ldots , r_{n}$). The algebra $R'$ is semisimple, so by Theorem \ref{ttt} it has a pseudo-homogeneous identity element $e+W=1$
 for some pseudo-homogeneous $e\in R$. Observe that for every $r\in R$ we have $r+W=1\cdot (r+W)=(r+W)\cdot {1}$, hence $r-er\in W$ and $r-re\in W$ for every $r\in R$.
 In particular $e^{2}-e\in W$.  
\end{proof}

\begin{theorem}\label{f} 
 Let $R$ be a finite dimensional $F$-algebra which is  not nilpotent and let notation be as in Definition \ref{17}. 
 Then there is $f\in R$ such that $f^{2}=f$  and for every $r\in R$ we have $r-fr\in W$ and $r-rf\in W$,
where $W$ is the Wedderburn radical of $R$.

Notice then that for every $r\in R$ we have $f(fr)=fr$, because $f^{2}=f$.
\end{theorem}
\begin{proof}
 Let  $e$ be as in the previous theorem. Then $e^{2}-e\in W$, and by the remark before Lemma \ref{ttt} there is $m>0$ such that $f=e^{m}$ satisfies $f^{2}=f$.
 Notice also that for every $r\in R$ we have $r-re^{m}\in W$. Indeed, the latter follows because $r-e^{m}r=(r-er)+(er-e^{2}r)+\ldots +(e^{m-1}r-e^{m}r)\in W$, since $r'-er'\in W$ where $r'=e^{i}r$,  by assumption.
 Similarly $r-e^{m} r\in W.$ 
\end{proof}  

\section{ Matrices and noncommutative algebras}

 In this section $F$ denotes the algebraic closure of a finite field.
 Let $N$ be a matrix whose coefficients are elements of $A$ of the same degree.
 Assume that for almost all $i$ entries of $N^{i}$ are  in $\langle x\rangle$.
  We will first show that  for some natural number $q$ either  $N^{q}=0$ or $M=N^{q}$ satisfies  Assumption $3$.
\begin{lemma}\label{matri}
 Let $N=\sum_{i=1}^{\xi '}A'_{i}a'_{i}$ where for each $i$, $A'_{i}$ is a matrix with coefficients in $F$ and $a'_{1}, a'_{2}, \ldots 
 a'_{\xi '}$ are elements from $A$ which are linearly independent over $F$ and have the same degree. Let $H'$ be the $F$-algebra generated by matrices $A'_{i}$. 
 Let  $\beta $ be a natural number. Then, for some $\xi $,  
  $N^{\beta }=\sum_{i=1}^{\xi  }A_{i}a_{i}$ where  each $a_{i}$ is a product of exactly $\beta $ elements from the set $\{a'_{0}, a'_{1}, \ldots , a'_{\xi }\}$. 
  Moreover, $a_{0}, a_{1}, \ldots , a_{\xi }$ are linearly independent over $F$
and $\{A_{1}, A_{2}, \ldots ,A_{\xi}\}$ equals the set of matrices which are products of exactly $\beta $ matrices from the set 
$\{A'_{1}, \ldots , A'_{\xi '}\}$
\end{lemma}
\begin{proof}
 Observe that distinct products $a'_{i_{1}} a'_{i_{2}}\cdots  a'_{i_{\beta }}$ are linearly independent over $F,$ because each of them has the same degree and each of them is a product  of elements  starting with $a$ or $b$---the generators of $R$.
 Therefore their products are linearly independent over $F$. Alternatively, it can be proved by induction on $\beta $.
\end{proof} 
\begin{lemma}\label{radical}
 Let notation be as in Lemma \ref{matri}, and let $H$ be an algebra generated by matrices $A_{1}, A_{2}, \ldots ,A_{\xi}$. Let $W$ be the Wedderburn radical of $H$, and $W'$ the Wedderburn radical of $H'$. If $r\in H\cap W'$ then $r\in W$.
\end{lemma}
\begin{proof} Observe first that the ideal generated by $r$ in $H'$ is nilpotent, as $r\in W'$. Therefore the ideal generated by $r$ in $H$ is nilpotent, and since $W$ is the sum of all nilpotent ideals in $H$ it follows that $r\in W$.
\end{proof}
\begin{theorem}\label{power}  
 Let notation be as in Lemmas \ref{matri} and \ref{radical}.  Assume that for almost all $i$ entries of $N^{i}$ are  in $\langle x\rangle$.
 Then for  infinitely many $\beta $ the matrix $M=N^{\beta }$ satisfies the following: either $M=0$ or 
 $M=\sum_{i=1}^{\xi }A_{i}a_{i}$, where for each $i$, $A_{i}$ is a matrix with entries in $F$ and $a_{1}, a_{2}, \ldots , 
 a_{\xi}$ are elements from $A\cap \langle x \rangle $ which are linearly independent over $F$ and have the same degree. Moreover, there is an element $e\in H$ which is a linear combination of matrices $A_{1}, \ldots , A_{\xi }$ and  such that $e^{2}=e$ and $r-er\in W$ and $r-re\in W$ for all $r\in H$, where $W$ is the Wedderburn radical of $H$.
\end{theorem}
\begin{proof}
 Recall that $N=\sum_{i=1}^{\xi '}A'_{i}a'_{i}$. If the algebra $H'$ generated by $A'_1, A'_2, \ldots , A'_{\xi '}$ is nilpotent then $N^{\beta }=0$ for almost all $\beta $. It remains to consider the case when $N$ is not a nilpotent matrix.

 We will first show that all the assertions of our theorem except  the assertion that $a_{1}, a_{2}, \ldots , 
 a_{\xi}$ are elements from $A\cap \langle x \rangle $ hold for some number $\beta $. 
Recall that  $A'_{1}, A'_{2}, \ldots A'_{\xi '}$ are generators of algebra $H'$. We can consider Definition \ref{17} for algebra $R=H'$ and for generators $r_{i}=A_{i}'$ for $i=1,\ldots , n$ where $n=\xi '$.
 We can then apply Theorem \ref{f} to algebra $H'$ to get that  there is $f\in H'$ such that $f^{2}=f$ and $f-fr=0$ and $r-rf=0$,  and $f$ is a pseudo-homogeneous  element of weight $\beta $, for an appropriate $\beta $. By Lemma \ref{matri} the set 
$\{A_{1}, A_{2}, \ldots ,A_{\xi}\}$ equals the set of elements which are products of exactly $\beta $ elements from the set 
$\{A'_{1}, \ldots , A'_{\xi '}\}$. Therefore 
 matrices $A_{1}, \ldots , A_{\xi }$ span the linear space of pseudo-homogeneous elements of weight $\beta $ in $H'$, hence $f$ is a linear combination of  $A_{1}, \ldots , A_{\xi }$; hence $f\in H$.
 Let $W'$ be the Wedderburn radical of $H'$. By Lemma \ref{radical}, $W'\cap H\subseteq W$, therefore $r-rf\in W$ and $r-fr\in W$ for all $r\in H$. Therefore  our result holds for  $e=f$ and $M=N^{\beta }$ (by Lemma \ref{matri}).
 We have shown that our result holds for some $\beta $.

To show that there are infinitely many elements $\beta $ with this property, observe that for any natural number $k>0$  we can apply the same reasoning to $\beta  _{k}=k\cdot \beta$ and $e_{k}=e^{k}$ instead of  $\beta $ and $e$.
 In this way we will obtain infinitely many $\beta \in \{\beta _{1}, \beta _{2} , \beta _{3} , \ldots \}$ satisfying the thesis.

Notice that $a_{1}, a_{2}, \ldots , 
 a_{\xi}$ are elements from $A\cap \langle x \rangle $ for sufficiently large $\beta $, this finishes the proof.
\end{proof}

\begin{corollary}\label{change} Let notation be as in Theorem \ref{power}. Then  we can assume that $A_{1}=e$ (by using linear combinations of elements $a_{i}$ instead of elements $a_{i}$).
\end{corollary}

\begin{corollary}\label{Assumption3} Let $A, A', R$ be as in Theorem \ref{Jacobson}, and $\langle x\rangle $ denote the ideal generated by $x$ in $A'$. Let $N$ be a matrix whose coefficients are elements of $A$ of the same degree.
 Assume that for almost all $i$ entries of $N^{i}$ are  in $\langle x\rangle$, then for some natural number $q$ either  $N^{q}=0$ or $M=N^{q}$ satisfies  Assumption $3$.
\end{corollary}
\begin{proof}
  It follows from Theorem \ref{power} and from Corollary \ref{change}.
\end{proof}

\begin{theorem}\label{assumption1}
  Assumption $1$ holds for $F$-algebra $A$.
\end{theorem}
\begin{proof} Let $N$ be a matrix with entries in $A^{*}(j)$ for some $j$, and such that for almost all $i $ matrix 
$N^{i }$ has all entries in $\langle x\rangle $.  By  Corollary \ref{Assumption3}  either $N$ is a nilpotent matrix or for some $q$  matrix $M=N^{q}$ satisfies Assumption $3$.
 By Theorem \ref{slm} applied to $M=N^{q}$, there are infinitely many $n$ such that the dimension of the space $R\cap S(L(M^{n}))$ doesn't exceed $\sqrt {n}$. Observe that 
 $S(L(M^{n}))=S(L(N^{q\cdot n}))$ (because operations $S$ and $L$ depend only upon the matrix, not on the way it is presented). Therefore $R\cap S(L(N^{q\cdot n}))$ has dimension $\leq \sqrt n\leq \sqrt {qn}$. This holds for infinitely many $n$.

\end{proof}

{\bf Proof of Theorem \ref{main}.}  
 Recall that $F$ is the  algebraic closure of a  finite field, and hence $F$ is countable and infinite.   
 By Theorem \ref{assumption1}, Assumption $1$ holds for $F$-algebra $A$.
 Then, by Theorem \ref{main2} there is an $F$-algebra $Z$ and a derivation $D$ on $Z$ such that the differential polynomial ring $Z[y; D]$ is Jacobson radical but $Z$ is not nil.

Assume now that $K$ is a subfield of $F$. If $R$ is an $F$-algebra then $R$ is also a $K$-algebra. 
 Therefore, Theorem \ref{main} also holds for an arbitrary subfield of the algebraic closure of any finite field. 

{\bf Acknowledgements} This research was supported by ERC Advanced grant Coimbra 320974. The author would like to thank Andr{\' e}  Leroy, Jason Bell and  Piotr Grzeszczuk for many useful comments. The author is very grateful to the unknown referee for 
 his/her many comments, which have improved the manuscript.

\end{document}